\documentclass[11pt,twoside]{article}
\usepackage{amsmath,amsbsy,amsfonts}
\usepackage{graphicx,color}
\usepackage{upgreek}
\everymath{\displaystyle}
\newtheorem{theorem}{Theorem}[section]
\newtheorem{lemma}[theorem]{Lemma}
\newtheorem{proposition}[theorem]{Proposition}
\newtheorem{corollary}[theorem]{Corollary}

\newtheorem{claim}[theorem]{Claim}

\newenvironment{proof}[1][Proof]{\noindent\textbf{#1.} }{\ \rule{0.5em}{0.5em}}

\newcommand{\R}{\mathbb{R}}
\newcommand{\Z}{\mathbb{Z}}
\newcommand{\N}{\mathbb{N}}
\newcommand{\M}{\mathcal{M}}
\newcommand{\K}{\mathcal{K}}

\begin{document}

\title{\bf Ground state solution for a class of indefinite variational problems with critical growth}

\author{Claudianor O. Alves\thanks{C. O. Alves was partially supported by CNPq/Brazil 301807/2013-2 and INCT-MAT, coalves@dme.ufcg.edu.br}\, , \, \ Geilson F. Germano \thanks{G. F. Germano was partially supported by CAPES, geilsongermano@hotmail.com}\vspace{2mm}
	\and {\small  Universidade Federal de Campina Grande} \\ {\small Unidade Acad\^emica de Matem\'{a}tica} \\ {\small CEP: 58429-900, Campina Grande - Pb, Brazil}\\}

\date{}
\maketitle

\begin{abstract}
In this paper we study the existence of ground state solution for an indefinite variational problem of the type  
$$
\left\{\begin{array}{l}
-\Delta u+(V(x)-W(x))u=f(x,u) \quad \mbox{in} \quad  \R^{N}, \\
u\in H^{1}(\R^{N}),
\end{array}\right. \eqno{(P)}
$$
where $N \geq 2$, $V,W:\mathbb{R}^N \to \mathbb{R}$ and $f:\mathbb{R}^N \times \mathbb{R} \to \mathbb{R}$ are continuous functions verifying some technical conditions and $f$ possesses  a critical growth. Here, we will consider the case where the problem is asymptotically periodic, that is, $V$ is $\mathbb{Z}^N$-periodic,  $W$ goes to 0 at infinity and $f$ is asymptotically periodic.

	\vspace{0.3cm}
	
	
	\vspace{0.3cm}
	
\end{abstract}

\section{Introduction}

In this paper we study the existence of ground state solution for an indefinite variational problem  of the type
$$
\left\{\begin{array}{l}
-\Delta u+(V(x)-W(x))u=f(x,u),\;\; \mbox{in} \;\; \R^{N}, \\
u\in H^{1}(\R^{N}),
\end{array}\right. \eqno{(P)}
$$
where $N \geq 2$, $V,W:\mathbb{R}^N \to \mathbb{R}$ are continuous functions verifying some technical conditions and $f$ has a critical growth. Here, we will consider the case where the problem is asymptotically periodic, that is, $V$ is $\mathbb{Z}^N$-periodic, $W$ goes to 0 at infinity and $f$ is asymptotically periodic.

In \cite{Szulkin1}, Kryszewski and Szulkin  have studied the existence of ground state solution for an indefinite variational problem of the type 
$$
\left\{\begin{array}{l}
-\Delta u+V(x)u=f(x,u), \quad \mbox{in} \quad  \R^{N}, \\
u\in H^{1}(\R^{N}),
\end{array}\right. \eqno{(P_1)}
$$
where $V:\mathbb{R}^{N} \to \mathbb{R}$ is a $\mathbb{Z}^N$-periodic continuous function such that 
$$
0 \not\in \sigma(-\Delta + V),\; \mbox{the spectrum of} \; -\Delta + V. \eqno{(V_1)}
$$
Related to the function $f:\mathbb{R}^N \times \mathbb{R} \to \mathbb{R}$, they assumed that $f$ is continuous, $\mathbb{Z}^N$-periodic in $x$ with
$$
|f(x,t)| \leq c(|t|^{q-1}+|t|^{p-1}), \quad \forall t \in \mathbb{R} \quad \mbox{and} \quad x \in \mathbb{R}^N \eqno{(h_1)}
$$ 
and
$$
0<\alpha F(x,t) \leq tf(x,t) \quad \forall t \in \mathbb{R} , \quad F(x,t)=\int_{0}^{t}f(x,s)\,ds \eqno{(h_2)}
$$
for some $c>0$, $\alpha >2$ and $2<q<p<2^{*}$ where $2^{*}=\frac{2N}{N-2}$ if $N \geq 3$ and $2^{*}=+\infty$ if $N=2$. The above hypotheses guarantee that the energy functional associated with $(P_1)$  given by 
$$
J(u)=\frac{1}{2}\int_{\mathbb{R}^N}(|\nabla u|^{2}+V(x)|u|^{2}\,dx)-\int_{\mathbb{R}^N}F(x,u)\,dx, \,\, u \in H^{1}(\R^N),
$$
is well defined and belongs to $C^{1}(H^{1}(\mathbb{R}^N), \mathbb{R})$. By $(V_1)$, there is an equivalent inner product $\langle \;\; , \;\; \rangle $  in $H^{1}(\mathbb{R}^N)$ such that 
$$
J(u)=\frac{1}{2}\|u^+\|^{2}-\frac{1}{2}\|u^-\|^{2} -\int_{\mathbb{R}^N}F(x,u)\,dx,
$$
where $\|u\|=\sqrt{\langle u,u \rangle}$ and  $H^{1}(\mathbb{R}^N) = E^{+} \oplus E^-$ corresponds to the spectral decomposition of $- \Delta + V $ with respect to the
positive and negative part of the spectrum with $u = u^{+}+u^{-}$, where $u^{+} \in E^{+}$ and $u^{-} \in E^{-}$.  In order to show the existence of solution for $(P_1)$, Kryszewski and Szulkin  introduced a new and interesting generalized link theorem. In \cite{LiSzulkin},  Li and Szulkin have improved this generalized link theorem to prove the existence of solution for a class of indefinite problem with $f$ being asymptotically linear at infinity.   

The link theorems above mentioned have been used in a lot of papers, we would like to cite   Chabrowski and Szulkin \cite{CS}, do \'O and Ruf \cite{DORUF},   Furtado and Marchi \cite{furtado}, Tang \cite{tang, tang2} and their references.

Pankov and Pfl\"uger \cite{Pankov-Pfluger} also have considered the existence of solution for problem $(P_1)$ with the same conditions considered in \cite{Szulkin1}, however the approach is based on an approximation technique of periodic function together with the linking theorem due to Rabinowitz \cite{Rabinowitz}. After, Pankov \cite{Pankov} has studied the  existence of solution for problems of the type  
$$
\left\{\begin{array}{l}
-\Delta u+V(x)u=\pm f(x,u), \quad \mbox{in} \quad  \R^{N}, \\
u\in H^{1}(\R^{N}),
\end{array}\right. \eqno{(P_2)}
$$
by supposing $(V_1)$, $(h_1)-(h_2)$ and employing the same approach explored in \cite{Pankov-Pfluger}.  In   \cite{Pankov} and \cite{Pankov-Pfluger}, the existence of ground state solution has been established by supposing that $f$ is $C^{1}$ and there is $\theta \in (0,1)$ such that
$$
0<t^{-1}f(x,t)\leq \theta f'_t(x,t), \quad \forall t \not=0 \quad \mbox{and} \quad x \in \mathbb{R}^N. \eqno({h_3})
$$
However, in  \cite{Pankov}, Pankov has found a ground state solution by minimizing the energy functional  $J$ on the set
$$
\mathcal{O}=\left\{u\in H^{1}(\R^{N})\setminus E^{-}\ ;\ J'(u)u=0\text{ and }J'(u)v=0,\forall\ v\in E^{-}\right\}. 
$$ 
The reader is invited to see that if $J$ is definite strongly, that is, when $E^{-}=\{0\}$,  the set $\mathcal{O}$ is exactly the Nehari manifold associated with $J$. Hereafter, we say that  $u_0 \in H^{1}(\mathbb{R}^{N})$ is called a {\it ground state solution} if 
$$
J'(u_0)=0, \quad u_0 \in \mathcal{O} \quad \mbox{and} \quad J(u_0)=\inf_{w \in \mathcal{O}}J(w).
$$

In \cite{SW}, Szulkin and Weth have established the existence of ground state solution for problem $(P_1)$ by completing the study made in  \cite{Pankov}, in the sense that, they also minimize the energy function 
on  $\mathcal{O}$, however they have used more weaker conditions on $f$, for example $f$ is continuous, $\mathbb{Z}^N$-periodic in $x$ and satisfies 
$$
|f(x,t)|\leq C(1+|t|^{p-1}), \;\; \forall t \in \mathbb{R} \quad \mbox{and} \quad x \in \mathbb{R}^N  \eqno{(f_1)}
$$ 
for some $C>0$ and $p \in (2,2^{*})$.
$$
f(x,t)=o(t) \,\,\, \mbox{uniformly in } \,\, x \,\, \mbox{as} \,\, |t| \to 0 \eqno{(f_2)}
$$
$$
F(x,t)/|t|^{2} \to +\infty \,\,\, \mbox{uniformly in } \,\, x \,\, \mbox{as} \,\, |t| \to +\infty \eqno{(f_3)}
$$
and
$$
t\mapsto f(x,t)/|t| \,\,\, \mbox{is strictly increasing on} \,\,\, \mathbb{R} \setminus \{0\}. \eqno{(f_4)}
$$

The same approach has been used by Zhang, Xu and Zhang \cite{ZXZ, ZXZ2} to study a class of indefinite and asymptotically periodic problem.

After a review bibliography, we have observed that there are few papers involving indefinite problem whose the nonlinearity has a critical growth. For example, the critical case for $N \geq 4$ was considered in \cite{CS}, \cite{Schechter2} and \cite{ZXZ2} when $f$ is given by
$$
f(x,t)=g(x,t)+k(x)|t|^{2^{*}-2}t,
$$
with $g:\mathbb{R}^N \times \mathbb{R} \to \mathbb{R}$ being a function with subcritical growth and $k:\mathbb{R}^{N} \to \mathbb{R}$ be a continuous function satisfying some conditions.  For the case $N=2$, we know only the paper \cite{DORUF} which considered the periodic case with $f$ having an exponential critical growth, namely there is $\alpha_0>0$ such that   
$$
\lim_{|t| \to +\infty}\frac{|f(t)|}{e^{\alpha|t|^{2}}}=0, \;\; \forall \alpha > \alpha_0, 
\lim_{|t| \to +\infty}\frac{|f(t)|}{e^{\alpha|t|^{2}}}=+\infty, \;\; \forall \alpha < \alpha_0.
$$

Motivated by ideas found in Szulkin and Weth \cite{SW, SW2} together with the fact that there are few papers involving critical growth for $N=2$ and $N \geq 3$ and indefinite problem, we intend in the present paper to study the existence of ground state solution for $(P)$ , with the nonlinearity $f$ having critical growth and the problem being asymptotically periodic. Since we will work with the dimensions $N=2$ and $N \geq 3$, we will state our conditions in two blocks, however the conditions on $V$ and $W$ are the same for any these dimensions. \\

\vspace{0.5 cm}

\noindent {\bf The conditions on $V$ and $W$}.

\vspace{0.5 cm}

On the  functions $V$ and $W$, we have assumed the following conditions: 
\begin{itemize}
	
	\item[$(V_1)$] $V:\R^{N}\rightarrow\R$ is continuous and $\Z^{N}$-periodic.
	\item[$(V_2)$] $\underline{\Lambda}:=\sup(\sigma(-\nabla+V)\cap(-\infty,0])<0<\overline{\Lambda}:=\inf(\sigma(-\nabla+V)\cap[0,+\infty))$.
	\item[$(W_1)$] $W:\R^{N}\rightarrow\R$ is continuous and $\displaystyle\lim_{|x|\rightarrow+\infty}W(x)=0$.
	\item[$(W_2)$] $\displaystyle 0<W(x)\leq \Theta=\sup_{x \in \R^{N}}W(x)<{\overline{\Lambda}}, \quad \forall x \in \mathbb{R}^N.$
	
\end{itemize}

With relation to function $f$, we have assumed the following conditions:

\vspace{0.5 cm}

\noindent {\bf The dimension $N \geq 3$:}

\vspace{0.5 cm}

For this case, we suppose that $f$ is the form
$$
f(x,t)=h(x)|t|^{q-1}t+k(x)|t|^{2^{*}-2}t
$$
with $1<q<2^{*}-1$ and 
\begin{itemize}
	\item[$(C_1)$] $h(x)=h_{0}(x)+h_{*}(x)$ and $k(x)=k_{0}(x)+k_{*}(x)$, where $h_{0},h_{*},k_{0},k_{*}:\R^{N}\rightarrow\R$ are continuous function, $h_0, k_0$ are $\Z^{N}$-periodic, $\displaystyle\lim_{|x|\rightarrow+\infty}h_{*}(x)=\displaystyle\lim_{|x|\rightarrow+\infty}k_{*}(x)=0$ and $h_{0}, h_{*},k_{0},k_{*}$ are positive.  
	\item[$(C_2)$] There is $x_0 \in \mathbb{R}^N$ such that  
	$$
k(x_0)=\max_{x \in \mathbb{R}^N}k(x) \quad \mbox{and} \quad k(x)-k(x_0)=o(|x-x_0|^{2}) \quad \mbox{as} \quad x \to x_0.
	$$
	\item[$(C_3)$] If $\inf_{x \in \mathbb{R}^N}h(x)=0$, we assume that $V(x_0)<0$. 
	
\end{itemize}	

\vspace{0.5 cm}

\noindent {\bf The dimension $N=2$:}

\vspace{0.5 cm}

\begin{itemize}
	\item[$(f_1)$] There exist functions $f_{0},f^{*}:\R^{2}\times\R\rightarrow\R$  such that 
	$$
	f(x,t)=f_{0}(x,t)+f^{*}(x,t),
	$$ 
	where $f_0$ and $f^*$ are continuous functions, $f_{0}$ is $\Z^{2}$-periodic with respect to $x$, $f^*$ is nonnegative and satisfies the following condition: Given $\epsilon>0$ and $\beta>0$, there exists $R>0$ such that 
	$$
	|f^{*}(x,t)|<\epsilon (e^{\beta s^{2}}-1) \quad \forall t \in \R \quad \mbox{and} \quad \forall  x\in\R^{2}\setminus B_{R}(0).
	$$ 
	\item[$(f_2)$] $\frac{f(x,t)}{t}, \frac{f_0(x,t)}{t}\rightarrow0 $ as $t\rightarrow0$ uniformly with respect to $x\in\R^{2}$.
\item[$(f_3)$] For each fixed $x\in\R^{2}$, the functions $t\mapsto\frac{f(x,t)}{t}$ and $t\mapsto\frac{f_{0}(x,t)}{t}$ are increasing on $(0,+\infty)$ and decreasing on $(-\infty,0)$. 
\item[$(f_4)$] There exist $\theta,\mu>2$ such that
$$0<\theta F_{0}(x,t)\leq tf_{0}(x,t)\hspace{1cm}\text{and}\hspace{1cm}0<\mu F(x,t)\leq tf(x,t)$$
for all $(x,t)\in\R^{2}\times\R$, where 
$$
F_{0}(x,t):=\int_{0}^{t}f_{0}(x,s)ds\ \ \ \text{ and }\ \ \ F(x,t):=\int_{0}^{t}f(x,s)ds.
$$
\item[$(f_5)$]\label{condicaocritica} There exists $\Gamma >0$ such that $|f_0(x,t)|,|f_*(x,t)|\leq \Gamma e^{4\pi t^2}$ for all $(x,t)\in\R^{2}\times\R$. 
\end{itemize}

\begin{itemize}
\item[$(f_6)$]$F_{0}(x,t)\geq D(x)|t|^{q},\ \forall\ (x,t)\in\R^{2}\times\R,$ \quad for some positive continuous function $D$ with $\inf_{x \in \mathbb{R}^2}D(x)>0$ and $q >2$. 
\end{itemize}

An example of a function $f$ verifying $(f_1)-(f_6)$ is 
$$
f(x,t)=\lambda(3-sen(x_1+x_2))|t|^{q-2}te^{4 \pi t^{2}}+\frac{1}{x_1^{2}+x_2^{2}+1}e^{4 \pi |t|^{\tau-2}t}, \quad \forall t \in \mathbb{R} 
$$
with $x=(x_1,x_2), \lambda >0, q \in (2,+\infty)$ and $\tau \in (1,2)$.

\vspace{0.5 cm}

The above conditions imply that $f$ has a critical growth  if $N=2$ or $N \geq 3$. 

\vspace{0.5 cm}

Our main theorem is the following:

\begin{theorem} \label{T1} Assume that $(V_1)-(V_2)$, $(W_1)-(W_2)$, $(C_1)-(C_3)$ and $(f_1)-(f_6)$ hold. Then, problem $(P)$ has a ground state solution for $N \geq 4$. If $N=2,3$, there is $\lambda^*>0$ such that if 
$ \inf_{x \in \mathbb{R}^2}D(x), 	\inf_{x \in \mathbb{R}^N}h(x) \geq \lambda^*$,  then problem $(P)$ has a ground state solution.

\end{theorem} 

\vspace{0.5 cm}

The Theorem \ref{T1} completes  the study made in some of the papers above mentioned, in the sense that we are considering others conditions on $V$ and $f$. For example, for the case $N \geq 3$, it completes the study made in \cite{SW}, because the critical case was not considered for $N \geq 3$ or $N=2$, and the case asymptotically periodic was not also analyzed. The Theorem \ref{T1} also  completes  \cite{DORUF}, because in that paper was proved the existence of a solution only for periodic case, while that we are finding ground state solution for the periodic and asymptotically periodic  case by  using a different method. Finally, the above theorem completes the main result of \cite{Schechter2} and \cite{ZXZ}, because the authors considered only the case $W=0$, and also the paper \cite{CS}, because the dimension $N=3$ was not considered as well as the asymptotically periodic  case. Moreover, in \cite{CS} and \cite{Schechter2} the authors considered only the
$$
V(x_0)<0 \quad \mbox{and} \quad k(x)-k(x_0)=o(|x-x_0|^{2}) \quad \mbox{as} \quad  x \to x_0.
$$      
In Theorem \ref{T1} this condition was not assumed if $\inf_{x \in \mathbb{R}^N}h(x)>0$. 

\vspace{0.5 cm}

Before concluding this introduction, we would like point out that the reader can find others interesting results involving  indefinite variational problem in Jeanjean \cite{Jeanjean},  Schechter \cite{Schechter,Schechter3}, Lin and Tang \cite{lin},  Willem and Zou \cite{WillemZou},  Yang \cite{Yang} and their references.  

\vspace{0.5 cm}

\noindent \textbf{Notation:} In this paper, we use the following
notations:
\begin{itemize}
	\item  The usual norms in $H^{1}(\R^N)$ and $L^{p}(\R^N)$ will be denoted by
	$\|\;\;\;\|_{H^{1}(\R^N)}$ and $|\;\;\;|_{p}$ respectively. 	

	\item   $C$ denotes (possible different) any positive constant.
	
	\item   $B_{R}(z)$ denotes the open ball with center  $z$ and
	radius $R$ in $\mathbb{R}^N$.
	
	\item We say that $u_n \to u$ in $L_{loc}^{p}(\mathbb{R}^N)$  when
	$$
	u_n \to u \quad \mbox{in} \quad L^{p}(B_R(0)), \quad \forall R>0.
	$$
 	\item  If $g$ is a mensurable function, the integral $\int_{\mathbb{R}^N}g(x)\,dx$ will be denoted by $\int g(x)\,dx$.	
	
\end{itemize}

The plan of the paper is as follows: In Section 2 we will show some technical lemmas and prove the Theorem \ref{T1} for $N \geq 3 $, while in Section 3 we will focus our attention to the dimension  $N=2$.

\section{The case $N \geq 3$}

In this section, our intention is to prove the Theorem \ref{T1} for the case $N \geq 3$. Some technical lemmas this section also are true for dimension $N=2$ and they will be used in Section 3.

In this section, our focus is the indefinite problem 
\begin{equation}\label{problema}\left\{\begin{array}{c}
-\Delta u+(V(x)-W(x))u=h(x) |u|^{q-1}u+k(x)|u|^{2^{*}-2}u,\;\; \mbox{in} \;\; \R^{N} \\
u\in H^{1}(\R^{N}), \\
\end{array}\right.
\end{equation}
whose the energy functional $\Phi_{W}:H^{1}(\R^{N})\rightarrow\R$  given by 
\begin{equation}
\Phi_{W}(u)=\frac{1}{2}B(u,u)-\frac{1}{2}\int W(x)|u|^{2}dx-\frac{1}{q+1}\int h(x){|u|^{q+1}}dx-\frac{1}{2^{*}}\int {k(x)|u|^{2^{*}}}dx
\end{equation}
is well defined, $\Phi_W \in C^{1}(H^{1}(\mathbb{R}^N), \mathbb{R})$ and  its critical points are precisely weak solutions of (\ref{problema}). Here, $B$ is the bilinear form
\begin{equation}\label{bilinear form}
B(u,v)=\frac{1}{2}\int (\nabla u\nabla v+V(x)uv)\, dx.
\end{equation}
  Note that the bilinear form $B$ is not positive definite, therefore it does not  induce a norm. As in \cite{SW}, there is an inner product $\langle \;\; , \;\; \rangle $  in $H^{1}(\mathbb{R}^N)$ such that 
\begin{equation} \label{PHIW}
  \Phi_W(u)=\frac{1}{2}\|u^+\|^{2}-\frac{1}{2}\|u^-\|^{2}-\frac{1}{2}\int W(x)|u|^{2}\,dx -\int F(x,u)\,dx,
 \end{equation}
  where $\|u\|=\sqrt{\langle u, u \rangle}$ and $H^{1}(\mathbb{R}^N) = E^{+} \oplus E^-$ corresponds to the spectral decomposition of $- \Delta + V $ with respect to the
  positive and negative part of the spectrum with $u = u^{+}+u^{-}$, where $u^{+} \in E^{+}$ and $u^{-} \in E^{-}$. It is well known that $B$ is positive definite on $E^{+}$, $B$ is negative definite on $E^{-}$ and the norm $\|\,\,\,\|$ is an equivalent norm to the usual norm in $H^{1}(\R^N)$, that is, there are $a,b >0 $ such that
  \begin{equation} \label{equivalente}
  b||u||^{2} \leq ||u||_{H^{1}(\R^{N})}^{2}\leq a||u||^{2},\ \ \forall\ u\in H^{1}(\R^{N}).
  \end{equation}

Hereafter, we denote by $\Phi:H^{1}(\R^{N})\rightarrow\R$ the functional defined by
$$
\Phi(u)=\frac{1}{2}B(u,u)-\frac{1}{q+1}\int h_{0}(x){|u|^{q+1}}dx-\frac{1}{2^{*}}\int k_{0}(x)|u|^{2^{*}}dx,
$$
or equivalently,
\begin{equation} \label{PHI}
\Phi(u)=\frac{1}{2}\|u^+\|^{2}-\frac{1}{2}\|u^-\|^{2}-\frac{1}{q+1}\int h_{0}(x){|u|^{q+1}}dx-\frac{1}{2^{*}}\int k_{0}(x)|u|^{2^{*}}dx.
\end{equation}

Note that the critical points of $\Phi$ are weak solutions of the periodic problem
\begin{equation}\label{problemacasoperiodico}
\left\{\begin{array}{c}
-\Delta u+V(x)u=h_{0}(x) |u|^{q-1}u+k_{0}(x)|u|^{2^{*}-2}u,\;\; \mbox{in} \;\; \R^{N}, \\
u\in H^{1}(\R^{N}). \\
\end{array}\right.
\end{equation}

In the sequel, $\M, E(u)$ and $\hat{E}(u)$ denote the following sets 
$$
\M:=\{u\in H^{1}(\R^{N})\setminus E^{-}\ ;\ \Phi_{W}'(u)u=0\text{ and }\Phi_{W}'(u)v=0,\forall\ v\in E^{-}\}
$$
and
$$
E(u):=E^{-}\oplus \R u\ \text{ and }\ \hat{E}(u):=E^{-}\oplus [0,+\infty)u.
$$ 
Therefore
$$
E(u)=E^{-}\oplus \R u^{+}\ \text{ and }\ \hat{E}(u)=E^{-}\oplus [0,+\infty)u^{+}.
$$ 
Moreover, we denote by $\gamma_W$ and $\gamma $ the real numbers
\begin{equation} \label{gamma}
\gamma_W:=\inf_{\M}\Phi_{W} \quad \mbox{and} \quad \gamma:=\inf_{\M}\Phi.
\end{equation}

\subsection{Technical lemmas}
In this section we are going to show some lemmas which will be used in the proof of main Theorem \ref{T1}. 

\begin{lemma}\label{contadomaximo}
	If $u\in\M$ and $w=su+v$ where $s\geq1$, $v\in E^{-}$  and $w\neq0$, then
	$$
	\Phi_{W}(u+w)<\Phi_{W}(u).
	$$
\end{lemma}
\begin{proof} In the sequel,  we fix
	$$
	G(x,t):=\frac{1}{2}W(x)t^{2}+\frac{1}{q+1}h(x)|t|^{q+1}+\frac{1}{2^{*}}k(x)|t|^{2^{*}}
	$$ 
	and 
	$$
	g(x,t):=W(x)t+h(x)|t|^{q-1}t+k(x)|t|^{2^{*}-2}t.
	$$	
Then by a simple computation,  
$$
\begin{array}{l}
\Phi_{W}(u+w)-\Phi_{W}(u)= \\
-\frac{1}{2}||v||^{2}+\int \left(g(x,u)\left[\left(\frac{s^{2}}{2}+s\right)u+(s+1)v\right] G(x,u)-G(x,u+w) \right) dx.
\end{array}
$$
Now, the proof follows by adapting  the ideas explored in \cite[Proposition 2.3]{SW}.
\end{proof}

\begin{lemma}\label{lemacompacto}
    Let $\K\subset E^{+}\setminus\{0\}$ be a compact subset, then there exists $R>0$ such that $\Phi_{W}(w)\leq0,\ \forall\ w\in E(u)\setminus B_{R}(0)$ and  $u\in\K$.
\end{lemma}
\begin{proof} Setting the functional 
$$
\Psi_*(u)=\frac{1}{2}B(u,u)-\frac{\lambda}{q+1}\int|u|^{q+1}\, dx 
$$
with $\lambda=\inf_{\R^{N}}h>0$, we have 
$$
\Phi_{W}(u) \leq \Psi_*(u), \quad \forall u \in H^{1}(\R^N). 
$$	
Now, we apply \cite[Lemma 2.2]{SW} to the functional $\Psi_*$ to get the desired result. 
\end{proof}
\begin{lemma}\label{supercontinuidade}
	For all $u\in H^{1}(\R^{N})$, the functional $\Phi_W|_{E(u)}$ is weakly upper semicontinuous.
\end{lemma}
\begin{proof}
	First of all, note that $E(u)$ is weakly closed, because it is convex strongly closed. Now, we claim that the functional
$$
	\begin{array}{rcccl}
	\widetilde{\Phi} & : & E(u) & \rightarrow & \R \\
	                    &   & w    & \mapsto     & \frac{1}{2}\int W(x)|w|^{2}\,dx+\frac{1}{q+1}\int h(x)|w|^{q+1}\,dx+\frac{1}{2^{*}}\int k(x)|w|^{2^{*}}\,dx
	                    \end{array}
$$
is weakly lower semicontinuous. Indeed, if $w_{n}\rightharpoonup w$ on $E(u)$, then after passing to a subsequence $w_{n}(x)\rightarrow w(x)$ a.e. in $\R^{N}$. Then  by Fatou's Lemma,
$$
\widetilde{\Phi}(w)=\int W(x)w^{2}\,dx+\frac{1}{q+1}\int h(x)|w|^{q+1}\,dx+\frac{1}{2^{*}}\int k(x)|w|^{2^{*}}\,dx\leq
$$
$$
\leq\liminf_{n\to +\infty}\left[\int W(x)w_{n}^{2}\,dx+\frac{1}{q+1}\int h(x)|w_{n}|^{q+1}\,dx+\frac{1}{2^{*}}\int k(x)|w_{n}|^{2^{*}}\,dx\right],
$$
leading to 
$$
\widetilde{\Phi}(w) \leq \liminf_{n\to +\infty}\widetilde{\Phi}(w_{n}).
$$
Furthermore, the functional
$$
\begin{array}{rcccl}
\widetilde{\Psi} & : & E(u) & \rightarrow & \R \\
&   & w    & \mapsto     & \frac{1}{2}B(w,w)
\end{array}
$$
is weakly upper semicontinuous. In fact, since
$$
\widetilde{\Psi}(w)=\frac{1}{2}(||w^{+}||^{2}-||w^{-}||^{2}),
$$
if $w_{n}=s_{n}u^{+}+v_n\rightharpoonup w=su^{+}+v$ with $v_{n},v\in E^{-}$, then $s_{n}\rightarrow s$ in $\R$ and $v_{n}\rightharpoonup v$ in $H^{1}(\R^N)$. Thus, 
$$
\widetilde{\Psi}(w)=\frac{1}{2}(s^{2}||u^{+}||^{2}-||v||^{2})\geq\limsup_{n \to +\infty}\frac{1}{2}(s_{n}^{2}||u^{+}||^{2}-||v_{n}||^{2})=\limsup_{n \to +\infty}\widetilde{\Psi}(w_{n}).
$$
As $\Phi_{W}|_{E(u)}=\widetilde{\Psi}-\widetilde{\Phi}$, the result is proved.              
\end{proof}

\begin{lemma}\label{funcaom}
	For each $u\in H^{1}(\R^{N})\setminus E^{-}$, $\M\cap\hat{E}(u)$ is a singleton set and the element of this set is the unique global maximum of $\Phi_{W}|_{\hat{E}(u)}$
\end{lemma}
\begin{proof}
	 The proof follows very closely the proof of \cite[Lemma 2.6]{SW}.  
\end{proof}
\begin{lemma}\label{primeirageometria}
There exists $\rho>0$ such that $\displaystyle\inf_{B_{\rho}(0)\cap E^{+}}\Phi_{W}>0$.
\end{lemma}
\begin{proof}
In what follows, let us fix $\overline{h}:=\sup_{x \in \R^{N}}h(x)$ and $\overline{k}:=\sup_{x \in \R^{N}}k(x)$. For $u\in E^{+}$, 
$$
\begin{array}{ll}
\Phi_{W}(u)= & \frac{1}{2}||u||^{2}-\frac{1}{2}\int W(x)|u|^{2}dx-\frac{1}{q+1}\int h(x)|u|^{q+1}dx-\frac{1}{2^{*}}\int k(x)|u|^{2^{*}}dx\\
& \geq\frac{1}{2}||u||^{2}-\frac{\Uptheta}{2}\int |u|^{2}dx-\frac{\overline{h}}{q+1}\int |u|^{q+1}dx-\frac{\overline{k}}{2^{*}}\int |u|^{2^{*}}dx \\
& \geq \frac{1}{2}||u||^{2}-\frac{ \Uptheta}{2\overline{\Lambda}}||u||^{2} 
-\frac{\overline{h}c_{1}}{q+1}||u||^{q+1}-\frac{\overline{k}c_{2}}{2^{*}}||u||^{2^{*}} \\
& =\frac{1}{2}\left(1-\frac{ \Uptheta}{\overline{\Lambda}}\right)||u||^{2}-\frac{\overline{h}c_{1}}{q+1}||u||^{q+1}-\frac{\overline{k}c_{2}}{2^{*}}||u||^{2^{*}}.
\end{array}
$$
Thereby, the lemma follows by taking $\rho>0$ satisfying 
$$
\frac{1}{2}\left(1-\frac{ \Uptheta}{\overline{\Lambda}}\right)\rho^{2}-\frac{\overline{h}c_{1}}{q+1}\rho^{q+1}-\frac{\overline{k}c_{2}}{2^{*}}\rho^{2^{*}}>0.
$$
\end{proof}
\begin{lemma}\label{cpositivo}
	The real number $\gamma_W$ given in (\ref{gamma}) is positive. In addition, if $u\in\M$ then $||u^{+}||\geq\max\{||u^{-}||,\sqrt{2\gamma_W}\}$.
\end{lemma}
\begin{proof}
By Lemma \ref{primeirageometria}, there is $\rho>0$ such that
$$
l:=\displaystyle\inf_{B_{\rho}(0)\cap E^{+}}\Phi_{W}>0.
$$
For all $u\in\M$, we know that $u^{+}\neq0$, then by Lemma \ref{funcaom},
$$
\Phi_{W}(u)\geq\Phi_{W}\left(\frac{\rho}{||u^{+}||}u^{+}\right)\geq l,
$$
from where it follows that
$$
\gamma_W =\inf_{\M}\Phi_{W}\geq l>0.
$$
In addition, for all $u\in\M$, 
$$
\gamma_W \leq\Phi_{W}(u)\leq\frac{1}{2}B(u,u)=\frac{1}{2}(||u^{+}||^{2}-||u^{-}||^{2}),
$$
implying that $||u^{+}||\geq\max\{||u^{-}||,\sqrt{2\gamma_W}\}$.
\end{proof}

\vspace{0.5 cm}

Next we will show a boundedness from above for $\gamma_W$ which will be crucial in our approach. However, before doing this we need to prove two technical lemmas. The first one is true for $N \geq 2$ and it has the following statement

\begin{lemma}\label{projecaolp}
	Consider $N \geq 2$ and let $u\in E^{+}\setminus\{0\}, p \in (2,2^{*})$ and $r, s_0>0$. Then there exists $\xi >0$ such that
	\begin{equation}\label{quaseimersao}
\xi	|su|_{p}\leq  |su+v|_{p},
	\end{equation}
	for all $s\geq s_{0}$ and $v\in E^{-}$ with  $||su+v||\leq r$.  
\end{lemma}
\begin{proof}	
If the lemma does not hold, there are $s_{n}\geq s_{0}$ and $v_{n}\in E^{-}$ satisfying 
$$
||s_{n}u+v_{n}||\leq r \;\;\; \mbox{and} \;\;\; |s_{n}u|_{p}\geq n|s_{n}u+v_{n}|_{p}, \,\, \forall n \in \mathbb{N}.
$$
Setting $\alpha_{n}:=|s_{n}u|_{p}$, we obtain 
$$
\left|\frac{u}{|u|_{p}}+\frac{v_{n}}{\alpha_{n}}\right|_{p}\leq\frac{1}{n}.
$$
Thus, passing to a subsequence if necessary,
\begin{equation}\label{limitew}
w_{n}:=\frac{v_{n}}{\alpha_{n}}\to -\displaystyle\frac{u}{|u|
	}_p\ \ \ \ \ \text{a.e. in} \quad \R^{N}.
\end{equation}
On the other hand,
$$
||w_{n}||^{2}=\frac{||v_{n}||^{2}}{s_{n}^{2}|u|_{p}^{2}}\leq \frac{||s_{n}u+v_{n}||^{2}}{s_{0}^{2}|u|_{p}^{2}}\leq\frac{r^{2}}{s_{0}^{2}|u|_{p}^{2}} \quad  \forall n \in \mathbb{N},
$$
showing that $(w_{n})$ is a bounded sequence in $H^{1}(\R^{N})$. As $w_{n}\in E^{-}$, there is $w\in E^{-}$ such that for some subsequence (not renamed) $w_{n}{\rightharpoonup}w$ in $E^{-}$. Then by (\ref{limitew}), 
$$
\displaystyle\frac{u}{|u|}_{p}=-w\in E^{-},
$$
which is absurd, since $u\in E^{+}\setminus\{0\}$.
\end{proof}

\begin{lemma}\label{lemares}
	Let $u\in E^{+}\setminus\{0\}$ be fixed. Then there are $r, s_0>0$ satisfying 
	\begin{equation}\label{igualdaderes}
	\sup_{w \in \widehat{E}(u)}\Phi_{W}(w)=\sup_{\scriptsize\begin{array}{c}  ||su+v||\leq r \\ s\geq s_{0}, v\in E^{-} \end{array}}\Phi_{W}(su+v).
	\end{equation}
\end{lemma}

\begin{proof}
   From Lemma \ref{lemacompacto}, 
$$
\sup_{\widehat{E}(u)}\Phi_{W}=\sup_{\widehat{E}(u)\cap B_{r}(0)}\Phi_{W}
$$
for some $r>0$. 	 Hence, there are $(s_{n}) \subset [0,+\infty)$ and $(v_{n}) \subset E^{-}$ with $||s_{n}u+v_{n}||\leq r$ and
\begin{equation}\label{supremoaproximado}
\Phi_{W}(s_{n}u+v_{n})\rightarrow \sup_{\widehat{E}(u)\cap B_{r}(0)}\Phi_{W}.
\end{equation}

    Next, we will prove that there exists $s_{0}>0$ such that
   	$$
   	\sup_{\widehat{E}(u)\cap B_{r}(0)}\Phi_{W}=\sup_{\scriptsize\begin{array}{c}  ||su+v||\leq r \\ s\geq s_{0}, v\in E^{-} \end{array}}\Phi_{W}(su+v).
   	$$
Arguing by contradiction, suppose that for all $s_{0}>0$ 
   	 \begin{equation}\label{desigualdadeestritasup}
   	 \sup_{\widehat{E}(u)\cap B_{r}(0)}\Phi_{W}>\sup_{\scriptsize\begin{array}{c}  ||su+v||\leq r \\ s\geq s_{0}, v\in E^{-} \end{array}}\Phi_{W}(su+v).
   	 \end{equation}
Such supposition permit us to conclude that $s_{n}\to 0$. On the other hand, recalling that  
$$
\Phi_{W}(s_{n}u+v_{n})\leq \frac{1}{2}s_{n}^{2}||u||^{2},
$$	
we are leading to  
$$
0< \gamma_W= \inf_{\M}\Phi_{W} \leq\sup_{\widehat{E}(u)}\Phi_{W}=\Phi_{W}(s_{n}u+v_{n})+o_n(1) \leq \frac{1}{2}s_{n}^{2}||u||^{2}+o_n(1),
$$
which is a contradiction. This completes the proof. \end{proof}

\vspace{0.5 cm}

Now, we are ready to show the estimate from above involving the number $\gamma_W$ given in (\ref{gamma})

\begin{proposition}\label{limitacaoc}
	Assume the conditions of Theorem \ref{T1}. If $N\geq4$, then 
	\begin{equation} \label{estimativa}
	\displaystyle \gamma_W <\frac{1}{N|k_{0}|_{\infty}^{\frac{N-2}{2}}}S^{N/2}.
	\end{equation}
	If $N=3$, there is $\lambda^* >0$ such that the estimate (\ref{estimativa}) holds for $\displaystyle\inf_{x \in \R^{N}}h(x)>\lambda^*$.  
	
\end{proposition}
\begin{proof} Since $\gamma_W \leq \gamma$, it is enough to prove that
$$
\displaystyle \gamma <\frac{1}{N|k_{0}|_{\infty}^{\frac{N-2}{2}}}S^{N/2}.
$$	
If $N \geq 4$ and $\inf_{x \in \mathbb{R}^N}h(x)=0$, the estimate is made in \cite[Proposition 4.2]{CS}. Next we will do the proof for $N \geq 4$ and $\inf_{x \in \mathbb{R}^N}h(x)>0$. To this end, we follow the same notation used in \cite{CS}. Let
$$
\varphi_{\epsilon}(x)=\frac{c_N\psi(x)\epsilon^{\frac{N-2}{2}}}{(\epsilon^2+|x|^{2})^{\frac{N-2}{2}}}
$$	
where $c_N=(N(N-2))^{\frac{N-2}{4}},  \epsilon>0$ and $\psi \in C_{0}^{\infty}(\mathbb{R}^N)$ is such that 
$$
\psi(x)=1 \quad \mbox{for} \quad |x| \leq \frac{1}{2} \quad \mbox{and} \quad \psi(x)=0 \quad \mbox{for} \quad |x| \geq 1.
$$
From \cite{W}, we know that the estimates below hold
\begin{equation}
\begin{array}{l} 
|\nabla \varphi_{\epsilon}|_{2}^{2}=S^{\frac{N}{2}}+O(\epsilon^{N-2}), \quad |\nabla \varphi_{\epsilon}|_{1}=O(\epsilon^{\frac{N-2}{2}}), \quad | \varphi_{\epsilon}|_{2^{*}}^{2^{*}}=S^{\frac{N}{2}}+O(\epsilon^{N}), \\
\mbox{}\\
|\varphi_{\epsilon}|_{2^*-1}^{2^*-1}=O(\epsilon^{\frac{N-2}{2}}), \quad |\varphi_{\epsilon}|_{q}^{q}=O(\epsilon^{\frac{N-2}{2}}), \quad |\varphi_{\epsilon}|_{1}=O(\epsilon^{\frac{N-2}{2}})
\end{array}
\end{equation}
and
\begin{equation}
|\varphi_{\epsilon}|_{2}^{2}=
\left\{
\begin{array}{l}
b\epsilon^{2}|log \epsilon|+ O(\epsilon^{2}), \quad \mbox{if} \quad \quad N=4 \\
b\epsilon^{2}+O(\epsilon^{N-2}), \quad \mbox{if} \quad N \geq 5.
\end{array}
\right.
\end{equation}

Adapting the same idea explored in \cite[Proposition 4.2]{CS}, for each $u \in E^{-}$ we obtain
$$
\Phi(s\varphi_{\epsilon}+u) \leq \Phi(s\varphi_{\epsilon}) + O(\epsilon^{N-2}), \quad \forall s \geq 0,
$$
where $O(\epsilon^{N-2})$ does not depend on $u$. Now, arguing as in \cite{Alves}, we get 
$$
\sup_{s \geq 0}\Phi(s\varphi_{\epsilon}) \leq \frac{1}{N|k_{0}|_{\infty}^{\frac{N-2}{2}}}S^{N/2} + O(\epsilon^{N-2})+c_1\int_{B_1(0)}|\varphi_{\epsilon}|^{2}\,dx-c_2\int_{B_1(0)}|\varphi_{\epsilon}|^{q+1}\,dx,
$$
implying that 
$$
\sup_{s \geq 0, \; u \in E^-}\Phi(s\varphi_{\epsilon}+u) \leq \frac{1}{N|k_{0}|_{\infty}^{\frac{N-2}{2}}}S^{N/2} +c_1\int_{B_1(0)}|\varphi_{\epsilon}|^{2}\,dx-c_2\int_{B_1(0)}|\varphi_{\epsilon}|^{q+1}\,dx + O(\epsilon^{N-2}). 
$$
Moreover, in \cite{Alves}, we also find that
$$
\lim_{\epsilon \to 0}\frac{1}{\epsilon^{N-2}}\left(c_1\int_{B_1(0)}|\varphi_{\epsilon}|^{2}\,dx-c_2\int_{B_1(0)}|\varphi_{\epsilon}|^{q+1}\,dx \right)=-\infty,
$$ 
from where it follows that there exists $\epsilon>0$ small enough verifying
$$
c_1\int_{B_1(0)}|\varphi_{\epsilon}|^{2}\,dx-c_2\int_{B_1(0)}|\varphi_{\epsilon}|^{q+1}\,dx + O(\epsilon^{N-2})<0, 
$$
and so,
$$
\sup_{s \geq 0, \; u \in E^-}\Phi(s\varphi_{\epsilon}+u) < \frac{1}{N|k_{0}|_{\infty}^{\frac{N-2}{2}}}S^{N/2}
$$
for some $\epsilon>0$ small enough.

Now, we will consider the case $N=3$. For each $u\in E^{+}\setminus\{0\}$, the Lemma \ref{lemares} guarantees the existence of $r,s_{0}>0$ satisfying 
$$
\sup_{w \in \widehat{E}(u)}\Phi(w)=\sup_{\scriptsize\begin{array}{c}  ||su+v||\leq r \\ s\geq s_{0}, v\in E^{-} \end{array}}\Phi(su+v).	
$$
Therefore, applying Lemma \ref{projecaolp}, 
$$
\begin{array}{ll}
\sup_{\widehat{E}(u)}\Phi= & \sup_{\scriptsize\begin{array}{c}  ||su+v|| 	 \leq r \\ s\geq s_{0}, v\in E^{-} \end{array}}\Phi(su+v) \\ 
 & \leq \sup_{\scriptsize\begin{array}{c}  ||su+v||\leq r \\ s\geq s_{0}, v\in E^{-} \end{array}}\left(\frac{s^{2}||u||^{2}}{2}-\frac{\lambda}{q+1}\int|su+v|^{q+1}dx\right) \\
& \leq\sup_{\scriptsize\begin{array}{c}  ||su+v||\leq r \\ s\geq s_{0}, v\in E^{-} \end{array}}\left(\frac{s^{2}||u||^{2}}{2}-\frac{\lambda \xi}{q+1}\int|su|^{q+1}dx\right) \\ 
&  \leq \max_{s \geq 0}(As^{2}-\lambda Bs^{q+1}),
\end{array}
$$
where 
$$
\lambda=\displaystyle\inf_{x \in \R^{N}}h(x), \quad  A=\frac{||u||^{2}}{2} \quad \mbox{and} \quad B=\frac{\xi}{q+1}\int|u|^{q+1}dx.
$$
As
$$
\max_{s \geq 0}(As^{2}-\lambda Bs^{q+1}) \to 0 \quad \mbox{as} \quad \lambda \to +\infty, 
$$
there is $\lambda^{*}>0$ such that
$$
\sup_{w \in \widehat{E}(u)}\Phi(w)  <\frac{1}{N|k_{0}|_{\infty}^{\frac{N-2}{2}}}S^{N/2} \quad \forall \lambda \geq \lambda^*,
$$
showing the desired result. \end{proof}

\begin{lemma}\label{limitacaosequenciaps}
Let  $(u_{n}) \subset H^{1}(\R^N)$ be a sequence verifying   
	$$
	\Phi_{W}(u_{n})\leq d,\ \ \ \pm\Phi_{W}'(u_{n})u_{n}^{\pm}\leq d||u_{n}||\ \ \ \text{and}\ \ \ \Phi_{W}'(u_{n})u_{n}\leq d||u_{n}||
	$$
for some $d>0$. Then, $(u_{n})$ is bounded in $H^{1}(\R^{N})$.
\end{lemma}
\begin{proof}
	In the sequel, let $\theta:=\chi_{[-1,1]}:\R\rightarrow\R$ be the characteristic function on interval $[-1,1]$,  
	$$
	g(x,t):=\theta(t)f(x,t)\quad \mbox{and} \quad j(x,t):=(1-\theta(t))f(x,t),
	$$
	where $f(x,t)=h(x)|t|^{q-1}t+k(x)|t|^{2^{*}-2}t$. Fixing 
	$$
	r:=\frac{q+1}{q} \quad \mbox{and} \quad s=\frac{2^{*}}{2^{*}-1},
	$$
	it follows that 
	$$
	(r-1)q=(s-1)(2^{*}-1)=1.
	$$
	Note that
	$$
	\begin{array}{ll}
	|g(x,t)|^{r-1} & =   \theta(t)^{r-1}|f(x,t)|^{r-1}\leq\theta(t)(|h|_{\infty}|t|^{q}+|k|_{\infty}|t|^{2^{*}-1})^{r-1}\ \\
	& \leq \theta(t) 2^{r-1}C(|t|^{(r-1)q}+|t|^{(r-1)(2^{*}-1)})\leq K|t|
	\end{array}
	$$
for some $C>0$ sufficiently large.  So
	\begin{equation}\label{desigualdadeg}
	|g(x,t)|^{r-1}\leq C|t|,\forall\ (x,t)\in\R^{N+1}. 
	\end{equation}
	Analogously, 
	\begin{equation}\label{desigualdadej}
	|j(x,t)|^{s-1}\leq C|t|, \forall\ (x,t)\in\R^{N+1}.
	\end{equation}
	Since $tf(x,t)\geq0,\ (x,t)\in\R^{N+1}$, the inequalities  (\ref{desigualdadeg}) and (\ref{desigualdadej}) give 
	\begin{equation}\label{desigualdadesgej}
	|g(x,t)|^{r}\leq C t g(x,t) \,\, \text{ and } \,\, |j(x,t)|^{s}\leq C t j(x,t),\ \ \ \forall (x,t)\in\R^{N+1}.
	\end{equation}
	The last two inequalities lead to
$$
\begin{array}{l}
	d+d||u_{n}||\geq\Phi_{W}(u_{n})-\frac{1}{2}\Phi_{W}'(u_{n})u_{n}= \\
	\left(\frac{1}{2}-\frac{1}{q+1}\right)\int h(x)|u|^{q+1}dx+\left(\frac{1}{2}-\frac{1}{2^{*}}\right)\int k(x)|u|^{2^{*}}dx  \geq \\ 
	
	\left(\frac{1}{2}-\frac{1}{q+1}\right)\int h(x)|u|^{q+1}dx+\left(\frac{1}{2}-\frac{1}{q+1}\right)\int k(x)|u|^{2^{*}}dx=\\
\left(\frac{1}{2}-\frac{1}{q+1}\right)\int (g(x,u_{n})u_{n}+j(x,u_{n})u_{n})dx\geq \\
	\left(\frac{1}{2}-\frac{1}{q+1}\right)\frac{1}{C}\left(\int
	|g(x,u_{n})|^{r}dx+\int |j(x,u_{n})|^{s}dx\right),
\end{array}
$$	
from where it follows 
	\begin{equation}
	|g(x,u_{n})|^{r}_{r}+|j(x,u_{n})|^{s}_{s}\leq C(1+||u_{n}||)
	\end{equation}
for some $C>0$. On the other hand,
	$$
\begin{array}{l}
	||u_{n}^{-}||^{2}=-\Phi'_{W}(u_{n})u_{n}^{-}-\int W(x)u_{n}u_{n}^{-}dx-\int f(x,u_{n})u_{n}^{-} \\
	\leq d||u_{n}^{-}||-\int W(x)u_{n}u_{n}^{-}+
	+|g(x,u_{n})|_{r}|u_{n}^{-}|_{q+1}+|j(x,u_{n})|_{s}|u_{n}^{-}|_{2^{*}} \\ 
\leq	-\int W(x)u_{n}u_{n}^{-}dx+C||u_{n}^{-}||\left(1+|g(x,u_{n})|_{r}+|j(x,u_{n})|_{s}\right) \\
\leq	-\int W(x)u_{n}u_{n}^{-}dx+C||u_{n}^{-}||\left(1+\left(1+||u_{n}||\right)^{1/r}+\left(1+||u_{n}||\right)^{1/s}\right) \\
\leq	-\int W(x)u_{n}u_{n}^{-}dx+C||u_{n}^{-}||\left(1+||u_{n}||^{1/r}+||u_{n}||^{1/s}\right).
\end{array}
$$
Thus,
	\begin{equation}\label{desigualdadeunmenos}
	||u_{n}^{-}||^{2}\leq-\int W(x)u_{n}u_{n}^{-}dx+C||u_{n}||\left(1+||u_{n}||^{1/r}+||u_{n}||^{1/s}\right).
	\end{equation}
The same argument works to prove that 
	\begin{equation}\label{desiguadadeunmais}
	||u_{n}^{+}||^{2}\leq\int W(x)u_{n}u_{n}^{+}dx+C||u_{n}||\left(1+||u_{n}||^{1/r}+||u_{n}||^{1/s}\right).
	\end{equation}
Recalling that $||u_{n}||^{2}=||u_{n}^{+}||^{2}+||u_{n}^{-}||^{2}$, the estimates 	(\ref{desigualdadeunmenos}) and (\ref{desiguadadeunmais}) combined give 
		\begin{equation}\label{desigualdadelimitacao}
	||u_{n}||^{2}\leq\int W(x)u_{n}(u_{n}^{+}-u_{n}^{-})dx+C||u_{n}||\left(1+||u_{n}||^{1/r}+||u_{n}||^{1/s}\right).
	\end{equation}
On the other hand, we know that
	$$
	\begin{array}{ll}
	\int W(x)u_{n}(u_{n}^{+}-u_{n}^{-})dx & =\int W(x)(u_{n}^{+}+u_{n}^{-})(u_{n}^{+}-u_{n}^{-})dx\\
	& =\int W(x)(u_{n}^{+})^{2}dx-\int W(x)(u_{n}^{-})^{2}dx \\
	& \leq\int W(x)(u_{n}^{+})^{2}dx\leq\Theta\int (u_{n}^{+})^{2}dx \leq \frac{ \Theta}{\overline{\Lambda}}||u_{n}^{+}||^{2}\\
	\end{array}
	$$
that is, 
	\begin{equation}\label{desigualdadeW}
	\int W(x)u_{n}(u_{n}^{+}-u_{n}^{-})dx\leq \frac{\Theta}{\overline{\Lambda}}||u_{n}||^{2},
	\end{equation}
where $\overline{\Lambda}$ was fixed in $(W_2)$. Now,  (\ref{desigualdadelimitacao}) combines with (\ref{desigualdadeW}) to give 
	$$
	\left(1-\frac{\Theta}{\overline{\Lambda}}\right)||u_{n}||^{2}\leq K_{6}||u_{n}||\left(1+||u_{n}||^{1/r}+||u_{n}||^{1/s}\right).
	$$
	This concludes the verification of Lemma \ref{limitacaosequenciaps}.
\end{proof}

As a byproduct of the last lemma, we have the corollaries below 
\begin{corollary} \label{LIM}
	If $(u_{n})$ is a (PS) sequence for $\Phi_{W}$, then $(u_{n})$ is bounded. In addition, if $u_{n}\rightharpoonup u$ in $H^{1}(\R^{N})$, then $u$ is a solution of $(\ref{problema}).$
\end{corollary}
\begin{corollary}\label{corolariocoercividadeemM}
	$\Phi_{W}$ is coercive on $\M$, that is, $\Phi_{W}(u)\rightarrow+\infty$ as $||u||\rightarrow+\infty$ and $u\in\M$.
\end{corollary}
The Lemma \ref{funcaom} permits to consider a function 
\begin{equation} \label{m}
m:E^{+}\setminus\{0\}\rightarrow\M\ \text{ where }\ m(u)\in\hat{E}(u)\cap\M, \quad \forall u\in E^{+}\setminus\{0\}.
\end{equation}

The above function will be crucial in our approach. Next, we establish its continuity. 

\begin{lemma}\label{funcaomcontinua}
	The function $m$ is continuous.
\end{lemma}

\begin{proof}
Suppose $u_{n} \to u$ in $E^{+}\setminus\{0\}$. Since 
$$
\frac{u_{n}}{||u_{n}||} \to \frac{u}{||u||},\quad  m\left(\frac{u_{n}}{||u_{n}||}\right)=m(u_{n}) \quad \mbox{and} \quad m\left(\frac{u}{||u||}\right)=m(u),
$$
without loss of generality, we may assume that $||u_{n}||=||u||=1$.

There are $t_{n}, t \in [0,+\infty)$ and $v_{n}, v \in E^{-}$ such that 
$$
m(u_{n})=t_{n}u_{n}+v_{n} \quad \mbox{and} \quad m(u)=tu+v.
$$
Note that $K:=\{u_{n}\}_{n\in\N}\cup\{u\}$ is a compact set. Thereby, by Lemma \ref{lemacompacto}, there exists $R>0$ such that $\Phi_{W}(w)\leq0$ in $E(z)\setminus B_{R}(0)$ for all $z\in K$.  Hence, 
$$
0<\Phi_{W}(m(u_{n}))=\sup_{\widehat{E}(u_n)}\Phi_{W}=\sup_{\widehat{E}(u_n)\cap B_{R}(0)}\Phi_{W}\leq\sup_{w\in\widehat{E}(u_n)\cap B_{R}(0)}\frac{1}{2}||w^{+}||^{2}\leq\frac{1}{2}R^{2},
$$
showing that $(\Phi_{W}(m(u_{n})))$ is a bounded sequence, and so, by Corollary \ref{corolariocoercividadeemM}, $(m(u_{n}))$ is a bounded sequence. The boundedness of $(m(u_{n}))$ implies that $(t_{n})$ and $(v_{n})$ are also bounded. Then, for some subsequence (not renamed),  
\begin{equation}\label{limitefracovn}
t_{n} \to  t_{0} \;\; \text{ in } \;\; \R, \,\,  v_{n} \rightharpoonup v_{0} \;\; \text{ in } \;\; E^{-} \quad \mbox{and} \quad  m(u_{n}) \rightharpoonup t_{0}u+v_{0} \;\; \text{ in } \;\; E^{-}.
\end{equation}
Recalling that $\Phi_{W}(m(u_{n}))\geq\Phi_{W}(tu_{n}+v)$, we obtain 
$$
\displaystyle\liminf_{n\to +\infty}\Phi_{W}(m(u_{n}))\geq\Phi_{W}(m(u)).
$$ 
Thus, the  Fatou's Lemma combined with the weakly lower semicontinuous of the norm gives  
$$
\begin{array}{ll}
\displaystyle\Phi_{W}(m(u))& \leq\liminf_{n\to +\infty}\Phi_{W}(m(u_{n}))\leq \limsup_{n\to +\infty}\Phi_{W}(m(u_{n}))  \\
& \limsup_{n\to +\infty}\left[\frac{1}{2}t_{n}^{2}||u_{n}||^{2}-\frac{1}{2}||v_{n}||^{2}-\frac{1}{2}\int W(x)m(u_{n})^{2}dx \right. \\
& \left. - \frac{1}{q+1}\int h(x)|m(u_{n})|^{q+1}dx 
-\frac{1}{2^{*}}\int k(x)|m(u_{n})|^{2^{*}}dx\right] \\
& \leq \frac{1}{2}t_{0}^{2}-\frac{1}{2}||v_{0}||^{2}-\vspace{0.3cm}
 \displaystyle-\frac{1}{2}\int W(x)|t_{0}u+v_{0}|^{2}dx\\
& -\frac{1}{q+1}\int h(x)|t_{0}u+v_{0}|^{q+1}dx-\frac{1}{2^{*}}\int k(x)|t_{0}u+v_{0}|^{2^{*}}dx\\
& =\Phi_{W}(t_{0}u+v_{0})\leq\Phi_{W}(m(u)),
\end{array}
$$
implying that 
\begin{equation}\label{conclusaom}
\lim_{n\to +\infty}||v_{n}||=||v_{0}|| \quad \mbox{and} \quad \Phi_{W}(t_{0}u+v_{0})=\Phi_{W}(m(u)).
\end{equation}
From (\ref{limitefracovn}) and (\ref{conclusaom}),  $v_{n} \to v_{0}$ in $E^{-}$. Now, the Lemma \ref{contadomaximo} together with (\ref{conclusaom}) guarantees that $t_{0}u+v_{0}=m(u)$. Consequently,
$$
m(u_{n})=t_{n}u_{n}+v_{n} \to t_{0}u+v_{0}=m(u),
$$
finishing the proof. 
\end{proof}

Hereafter, we consider the functional $\hat{\Psi}:E^{+}\setminus\{0\}\rightarrow\R$ defined by $\hat{\Psi}(u):=\Phi_{W}(m(u))$. We know that $\hat{\Psi}$ is continuous by previous lemma. In the sequel, we denote by $\Psi:S^{+}\rightarrow\R$ the restriction  of $\hat{\Psi}$ to $S^{+}=B_1(0) \cap E^+$. 

The next three results establish some important properties involving the functionals $\Psi$ and $\hat{\Psi}$ and their proofs follow as in \cite{SW}. 

\begin{lemma}\label{funcaopsidiferenciavel}
	$\hat{\Psi}\in C^{1}(E^{+}\setminus\{0\},\R)$, and
	\begin{equation}\label{igualdadediferencial}
	\hat{\Psi}'(y)z=\frac{||m(y)^{+}||}{||y||}\Phi_{W}'(m(y))z,\ \forall y,z\in E^{+},\ y\neq 0.
	\end{equation}
\end{lemma}

\begin{corollary}\label{corolariosequenciasps}
	The following assertions hold:
	\begin{itemize}
		\item[(a)] $\Psi\in C^{1}(S^{+})$, and
		$$
		\Psi'(y)z=||m(y)^{+}||\Phi_{W}'(m(y))z,\text{ for }z\in T_{y}S^{+}.
		$$
		\item[(b)] $(w_{n})$ is a (PS)$_{c}$ sequence for $\Psi$ if and only if $(m(w_{n}))$ is a (PS)$_{c}$ sequence for $\Phi_{W}$.
		\item[(c)] If $\gamma_W=\inf_{\M}\Phi_{W}$ is attained by $u\in\M$, then $\Phi_{W}'(u)=0$.
	\end{itemize}
\end{corollary}
\begin{proposition}\label{pssequence}
	There exists a (PS)$_{\gamma_W}$ sequence for $\Phi_{W}$.
\end{proposition}

Our next lemma will be used to prove the existence of ground state solution for the periodic case. 

\begin{lemma}\label{finalperiodico}
Let $(u_{n})$ be a (PS)$_{c}$ sequence for functional $\Phi$ given in (\ref{PHI}) with $c\neq0$. Then, there are $r,\epsilon>0$ and $(y_{n})$ in $\Z^{N}$ satisfying
	\begin{equation}\label{negacaohipoteselions}
	\limsup_{n\in\N}\int_{B_{r}(y_{n})}|u_{n}|^{2^{*}}dx\geq\epsilon.
	\end{equation}
	In addition, if $c\in (-\infty,S^{N/2}|k_0|_{\infty}^{\frac{2-N}{2}}/N)\setminus\{0\}$, the sequence $v_{n}=u_{n}(\cdot-y_{n})$ is also a $(PS)_c$ sequence for $\Phi$, and for some subsequence, $v_{n}\rightharpoonup v$ in $H^1(\R^N)$ with $v\neq0$.
\end{lemma}
\begin{proof}
	By Corollary \ref{LIM}, the sequence $(u_{n})$ is bounded in $H^{1}(\R^N)$. Arguing by contradiction, we suppose that
	$$
	\limsup_{n\to +\infty}\sup_{y\in\R^{N}}\int_{B_{R}(y)}|u_{n}|^{2^{*}}dx=0,
	$$
for some $R>0$. Applying \cite[Lemma 2.1]{RWW}, it follows that $u_{n}\rightarrow0$ in $L^{2^{*}}(\R^{N})$, and so, by interpolation  on the Lebesgue spaces,  $u_{n}\rightarrow 0$ in $L^{p}(\R^{N})$ for all $p\in(2,2^{*}]$. As
$$
\Phi'(u_{n})(u_{n}^{-})=-||u_{n}^{-}||^{2}-\int h_{0}(x)|u_{n}|^{q-1}u_{n}u_{n}^{-}dx-\int k_{0}(x)|u_{n}|^{2^{*}-2}u_{n}u_{n}^{-}dx,
$$
we deduce that $u_{n}^{-} \to 0$ in $H^{1}(\R^{N})$. By a similar argument $u_{n}^{+} \to 0$ in $H^{1}(\R^{N})$. Hence	
$$
u_{n}\rightarrow0 \text{ in }\;\; H^{1}(\R^{N}).
$$
Thereby, by continuity of $\Phi$, $c=\lim\Phi(u_{n})=\Phi(0)=0$, which is absurd. Thus, there are $(z_n) \subset \R^N$ and $\eta>0$ satisfying
$$
\int_{B_{R}(z_{n})}|u_{n}^{+}|^{2^*}dx\geq\eta>0,\ \ \ \forall n\in\N.
$$
Recalling that for each $n \in \mathbb{N}$ there is $y_n \in \Z^N$  such that 
$$
B_R(z_n) \subset B_{R+\sqrt{N}}(y_n), 
$$
we have  
$$
\int_{B_{R+\sqrt{N}}(y_{n})}|u_{n}^{+}|^{2^*}dx\geq\eta>0,\ \ \ \forall n\in\N,
$$
finishing the proof of (\ref{negacaohipoteselions}).  
	
Now, assume  $c\in (-\infty,S^{N/2}|k_0|_{\infty}^{\frac{2-N}{2}}/N)\setminus\{0\}$ and set $v_{n}:=u_{n}(\cdot -y_{n})$. By a simple computation, we see that $(v_{n})$ is also a $(PS)_{c}$ sequence for $\Phi$ with 
\begin{equation}\label{lions2*}
\limsup_{n\to +\infty}\int_{B_{r}(0)}|v_{n}^+|^{2^{*}}\,dx \geq\epsilon.
\end{equation}
By Corollary \ref{corolariocoercividadeemM}, $(v_{n})$ is bounded, and so, for some subsequence ( sill denoted by $(v_n)$ ), $v_{n}\rightharpoonup v$ in $H^{1}(\R^N)$ for some $v \in H^{1}(\R^N)$. Suppose by contradiction  $v=0$ and assume that 
	\begin{equation}\label{Lions}
	|\nabla v_{n}|^{2}\rightharpoonup\mu\hspace{1cm}\text{ and }\hspace{1cm}|v_{n}|^{2^{*}}dx\rightharpoonup\nu \text{ in }\M^{+}(\R^{N}).
	\end{equation}
By Concentration-Compactness Principle II due to Lions \cite{lionsII}, there exist a countable set $\mathcal{J}$, $(x_{j})_{j\in \mathcal{J}} \subset \R^{N}$ and $(\mu_{j})_{j\in \mathcal{J}},(\nu_{j})_{j\in \mathcal{J}} \subset [0,+\infty)$ such that
\begin{equation}\label{pcc}
\nu=\sum_{j\in \mathcal{J}}\nu_{j}\delta_{x_{j}} \quad \mu \;\; \geq\sum_{j\in \mathcal{J}}\mu_{j}\delta_{x_{j}}\;\; \text{ with } \;\; \mu_{j} \geq S\nu_{j}^{\frac{2}{2^{*}}}.
\end{equation}
Now, our goal is showing that $\nu_{j}=0$ for all $j\in \mathcal{J}$. First of all, note that 
\begin{equation}\label{desigualdadec}
c=\lim_{n \to +\infty}\left[\Phi(v_{n})-\frac{1}{2}\Phi'(v_{n})v_{n}\right] \geq \frac{1}{N}\sum_{j\in \mathcal{J}}k_{0}(x_{j})\nu_{j}.
\end{equation}

On the other hand, setting $\psi_{\epsilon}(x):=\psi((x-x_{j})/{\epsilon}),\forall\ x\in\R^{N},\forall\ \epsilon>0$, where 	$\psi\in C^{\infty}_{c}(\R^{N})$ is such that $\psi\equiv1$ in $B_{1}(0)$, $\psi\equiv0$ in $\R^{N}\setminus B_{2}(0)$ and $|\nabla\psi|\leq2$, with $0\leq \psi \leq1$, we have that $\psi_{\epsilon}v_{n}\in H^{1}(\R^{N})$ and $(\psi_{\epsilon}v_{n})$ is bounded in $H^{1}(\R^{N})$. So
$$
\Phi'(v_{n})(\psi_{\epsilon}v_{n})\rightarrow0
$$
or equivalently
$$
\int\nabla v_{n}\nabla(\psi_{\epsilon}v_{n})\,dx+\int V(x)\psi_{\epsilon}v_{n}^{2}\,dx-\int h_{0}(x)\psi_{\epsilon}|v_{n}|^{q+1}dx-\int k_{0}(x)|v_{n}|^{2^{*}}\psi_{\epsilon}dx\to 0.
$$
By using the definition of $\nu$ and $\mu$ together with the last limit, we derive 
$$
\int\nabla v(\nabla\psi_{\epsilon})v\,dx+\int V(x)\psi_{\epsilon}v^{2}\,dx-\int h_{0}(x)\psi_{\epsilon}|v|^{q+1}dx+\int\psi_{\epsilon}d\mu-\int k_{0}\psi_{\epsilon}d\nu=0.
$$
Now, taking the limit $\epsilon\rightarrow0$,  we find
$$
\mu(x_{j})=k_{0}(x_{j})\nu_{j}.
$$
By (\ref{pcc}), $\mu_{j}\leq\mu(x_{j})$. Then,   
$$
S\nu_{j}^{2/(2^{*})}=\mu_{j}\leq\mu(x_{j})=k_{0}(x_{j})\nu_{j}.
$$ 
If $\nu_{j}\neq0$, the last inequality gives
\begin{equation} \label{A1}
\nu_{j}\geq\frac{S^{N/2}}{|k_0|_{\infty}^{\frac{N-2}{2}}}.
\end{equation}
Thereby, by (\ref{desigualdadec}) and (\ref{A1}), if there exists $j\in \mathcal{J}$ such that $\nu_{j}\neq0$, we would have   
$$
c\geq\frac{S^{N/2}}{N|k_0|_{\infty}^{\frac{N-2}{2}}}
$$ 
which is absurd. Hence $\nu_{j}=0$ for all $j\in \mathcal{J}$, so $\nu\equiv0$, and by (\ref{Lions}), $|v_{n}|^{2^{*}}\rightharpoonup0$ in $\M^{+}(\R^{N})$. Consequently $v_{n}\rightarrow0$ in $L^{2^{*}}_{loc}(\R^{N})$ which contradicts  (\ref{lions2*}), showing that $v\neq0$.
\end{proof}

\subsection{Proof of Theorem \ref{T1}: The case $N \geq 3$. }

The proof will be divided into two cases, more precisely, the Periodic Case and  the  Asymptotically Periodic Case. 

\vspace{0.5 cm}

\noindent {\bf 1-  The Periodic Case:}

\vspace{0.5 cm}

\begin{proof}
	From Proposition \ref{pssequence}, there exists a   $(PS)_{\gamma}$ sequence $(u_{n})$ for $\Phi$, where $\gamma$ was given in (\ref{gamma}). By  Lemma \ref{finalperiodico}, passing to a subsequence if necessary, $u_{n}\rightharpoonup u\neq0$ and $u\in H^{1}(\R^{N})$ is a solution of problem (\ref{problemacasoperiodico}), and so,  $\Phi(u)\geq \gamma$. On the other hand  
$$
\gamma=\lim_{n\to +\infty}\left[\Phi(u_{n})-\frac{1}{2}\Phi'(u_{n})(u_{n})\right]=\liminf_{n\to +\infty}\left[\left(\frac{1}{2}-\frac{1}{q+1}\right)\int h(x)|u_{n}|^{q+1}dx\right.
$$
$$\left.+\left(\frac{1}{2}-\frac{1}{2^{*}}\right)\int k(x)|u_{n}|^{2^{*}}\,dx \right]\geq\left[\left(\frac{1}{2}-\frac{1}{q+1}\right)\int h(x)|u|^{q+1}dx\right.+
$$
$$
\left.+\left(\frac{1}{2}-\frac{1}{2^{*}}\right)\int k(x)|u|^{2^{*}}\,dx \right]=\Phi(u)-\frac{1}{2}\Phi'(u)u=\Phi(u).
$$
	From this, $u\in H^{1}(\R^{N})$ is a ground state solution for the problem $(\ref{problemacasoperiodico})$.
\end{proof}

\vspace{0.5 cm}

\noindent {\bf 2- Asymptotically Periodic Case }

\vspace{0.5 cm}

\begin{proof} From definition of $\Phi_{W}$ and $\Phi$,  we have the inequality 
$$
\gamma_W \leq \gamma.
$$
Next, our analysis will be divide into two cases, more precisely, $\gamma_W=\gamma$ and $\gamma_W < \gamma$.  

Assume firstly $\gamma_W =\gamma$. Let $u\in H^{1}(\R^{N})$ be a ground state solution of (\ref{problemacasoperiodico}) for the periodic case and  $v\in\widehat{E}(u)$ such that  
$$
\Phi_{W}(v)=\displaystyle\sup_{\widehat{E}(u)}\Phi_{W}.
$$  
Then,
$$
\gamma_W \leq\Phi_{W}(v)\leq\Phi(v)\leq\Phi(u)=\gamma=\gamma_W,
$$
implying that $\Phi_{W}(v)=\gamma_W $ with $v\in\M$. By Corollary \ref{corolariosequenciasps}, part (c), we deduce that $v$ is a ground state solution of (\ref{problema}).

Now, assume $\gamma_W <\gamma$ and let $(u_{n})$ be a $(PS)_{\gamma_{W}}$ sequence for $\Phi_{W}$ given by Proposition \ref{pssequence}. By Lemma \ref{limitacaosequenciaps}, $(u_{n})$ is a bounded sequence, then for some subsequence (still denoted by $(u_n)$)  $u_{n}\rightharpoonup u$ in $H^{1}(\R^N)$. We claim that $u \not=0$. Indeed, if $u=0$ it is easy to see  that 
$$
\int W(x)u_{n}^{2}dx \to 0 \;\; \mbox{and} \;\; \sup_{\|\psi\|\leq 1}\left|\int W(x) u_{n}\psi dx \right| \to 0.
$$ 
In addiction, by $(C_1)$, we also have 
$$
\int h^{*}(x)|u_{n}|^{q+1} dx \to 0 \quad \mbox{and} \quad \sup_{\|\psi\|\leq 1}\left|\int h^{*}(x)|u_{n}|^{q-1}u\psi dx \right| \to 0.
$$ 
Arguing as in Lemma \ref{finalperiodico}, we derive that $u_{n}\rightarrow 0$ in $L^{2^{*}}_{loc}(\R^{N})$, and so, 
$$
\int k^{*}(x)|u_{n}|^{2^{*}}dx \to 0 \quad \mbox{and} \quad  \sup_{\|\psi\|\leq 1}\left|\int k^{*}(x)|u_{n}|^{2^{*}-2}u_n\psi dx\right| \to 0. 
$$
Hence 
$$
\Phi_W(u_{n}) \to \gamma_{W} \;\; \mbox{and} \;\;  ||\Phi_W'(u_{n})|| \to 0,
$$
that is, $(u_{n})$ is a $(PS)_{\gamma_W}$ sequence for $\Phi_W$. By Proposition \ref{limitacaoc},  
$$
\gamma_W <\frac{S^{N/2}}{N|k_{0}|_{\infty}^{\frac{N-2}{2}}}.
$$ 
Then, Proposition \ref{finalperiodico} guarantees the existence of  $(y_{n}) \subset \Z^{N}$ such that $v_{n}:=u_{n}(\cdot-y_{n})\rightharpoonup v\neq0$ in $H^{1}(\R^N)$ and $\Phi'(v)=0$. Consequently 
$$
\begin{array}{ll}
\gamma_W & = \lim_{n\to +\infty}\Phi_{W}(u_{n})=\lim_{n \to +\infty}\Phi(u_{n}) \\
  & = \lim_{n\to +\infty}\Phi(v_{n})=\lim_{n\to +\infty}\left[\Phi(v_{n})-\frac{1}{2}\Phi'(v_{n})v_{n}\right] \\
 & \geq \Phi(v)-\frac{1}{2}\Phi'(v)v=\Phi(v)\geq \gamma
\end{array}
$$
which is absurd, proving that $u\neq0$. Now, we repeat the same argument explored in the periodic case to conclude that $u$ is a ground state solution of (\ref{problema}). \end{proof}

\section{The case $N=2$ }

In this section we are going to show the existence of ground state solution for the following indefinite problem
\begin{equation}\label{probleman2}
\left\{\begin{array}{l}
-\Delta u+(V(x)-W(x))u=f(x,u),\quad \mbox{in} \,\, \R^{2}, \\
u\in H^{1}(\R^{2}),
\end{array}\right.
\end{equation}
by assuming $(V_1),(V_2), (W_1),(W_2)$ and $(f_1)-(f_6)$. Since we will work with exponential critical growth, in the next subsection we recall some facts involving this type of growth.

\subsection{Results involving exponential critical growth}

\hspace{0.5 cm} The exponential critical growth on $f$ is motivated by the following estimates proved by Trudinger \cite{trudinger} and Moser \cite{moser}.

\begin{lemma} \noindent {\bf (Trudinger-Moser inequality for bounded domains)}\label{lema 1.1}
	Let $\Omega \subset \R^2$ be a bounded domain. Given any  $ u \in H_0^{1}(\Omega)$, we have
	$$
	\int_{\Omega}
	e^{\alpha\left|u\right|^{2}}dx
	< \infty, \,\,\,\, \mbox{ for every }\,\,\alpha >0.
	$$
	Moreover, there exists a positive constant $C=C(|\Omega|)$ such that
	\[
	\sup_{||u||\leq 1} \int_{\Omega} e^{\alpha|u|^{2}} dx \leq C , \,\,\,\,\,\,\, \mbox{for all } \, \alpha  \leq 4 \pi,
	\]
\end{lemma}

The next result is a version of the Trudinger-Moser inequality for whole $\mathbb{R}^{2}$, and its proof can be found in Cao \cite{Cao} ( see also Ruf \cite{Ruf} ).

\begin{lemma} \noindent {\bf (Trudinger-Moser inequality for unbounded domains)}\label{lema 1.2}
	For all $ u \in H^{1}(\mathbb{R}^{2})$, we have
	$$
	\int 
	\left(e^{\alpha\left|u\right|^{2}}-1 \right)dx
	< \infty,\,\,\,\, \mbox{ for every }\,\,\alpha >0.
	$$
	Moreover, if $\left| \nabla
	u\right|^{2}_{2}\leq 1,\,\left|u\right|_{2} \leq M < \infty $ and
	$\alpha < 4 \pi$, then there
	exists a positive constant $C=C(M,\alpha)$ such that
	$$
	\int \left(e^{\alpha\left|u\right|^{2}}-1\right)dx \leq C,
	$$
	\end{lemma}

The Trudinger-Moser inequalities will be strongly utilized throughout
this section in order to deduce important estimates. The reader can find more recent results involving this inequality in  \cite{CT},  \cite{masmoudi}, \cite{ish}, \cite{MS} and references therein

In the sequel, we state
some technical lemmas found in \cite{AL} and \cite{djairo}, which will be essential to carry out the proof of our
results.

\begin{lemma}\label{lema 1.3} Let  $ \alpha > 0$ and $t \geq 1$. Then, for every each $\beta > t$, there exists a constant $C=C(\beta,t)
	> 0$ such that
	$$
	\left(e^{4\pi \left|s\right|^{2}}-1 \right)^{t} \leq C
	\left(e^{\beta 4 \pi \left|s\right|^{2}}-1\right), \quad \forall s \in \mathbb{R}.
	$$
\end{lemma}

\begin{lemma} \label{LX} Let $(u_n)$ be a sequence such that $u_n(x) \to u(x)$ a.e. in $\R^2$ and $(f(x,u_n)u_n)$ is bounded in $L^{1}(\R^{2})$. Then, $f(x,u_n) \to f(x,u)$ in $L^1(B_R(0))$ for all $R>0$, and so,  
$$
\int f(x,u_n)\phi \, dx  \to \int f(x,u)\phi \, dx, \quad \forall \phi \in C^{\infty}_{0}(\R^2).
$$	
\end{lemma}

\subsection{Technical Lemmas} 

In this subsection we have used the same notations of Section 2, however we will recall some of them for the convenience of the reader.  In what follows,  we denote by $\Phi_{W}:H^{1}(\R^{2})\rightarrow\R$ the energy functional given by 
$$
\Phi_{W}(u):=\frac{1}{2}B(u,u)-\frac{1}{2}\int W(x)|u|^{2}dx-\int F(x,u)dx,
$$
where $B:H^{1}(\R^{2})\times H^{1}(\R^{2})\rightarrow\R$ is the bilinear form
$$
B(u,v):=\int(\nabla u\nabla v+V(x)uv)dx,\ \ \forall\ u,v\in H^{1}(\R^{2}).
$$
It is well known that $\Phi_W \in C^{1}(H^{1}(\R^2), \R)$ with 
$$
\Phi_{W}'(u)v=B(u,v)-\int W(x)uv dx-\int f(x,u)v dx, \quad \forall u,v \in H^{1}(\R^2). 
$$
Therefore critical points of $\Phi_{W}$ are solutions of (\ref{probleman2}). Moreover, we can rewrite the functional $\Phi_W$ of the form
$$
\Phi_W(u)=\frac{1}{2}\|u^+\|^{2}-\frac{1}{2}\|u^-\|^{2}-\frac{1}{2}\int W(x)|u|^{2}\,dx -\int F(x,u)\,dx,
$$

In what follows, we also consider the $C^{1}$-functional $\Phi:H^{1}(\R^{2})\rightarrow\R$
$$
\Phi(u):=\frac{1}{2}B(u,u)-\int F_{0}(x,u)dx
$$
or equivalently
$$
\Phi(u)=\frac{1}{2}\|u^+\|^{2}-\frac{1}{2}\|u^-\|^{2}-\int F_0(x,u)\,dx,
$$
whose the critical points are weak solutions of periodic problem
\begin{equation}\label{problemacasoperiodicon2}
\left\{\begin{array}{l} -\Delta u+V(x)=F_{0}(x,u),\quad \mbox{in} \;\; \R^{2}, \\ u\in H^{1}(\R^{2})\end{array}\right.
\end{equation}

As in Section 2, we will consider the sets
$$
\M:=\{u\in H^{1}(\R^{2})\setminus E^{-}\ ;\ \Phi_{W}'(u)u=0\text{ and }\Phi_{W}'(u)v=0,\forall\ v\in E^{-}\},
$$
$$
E(u):=E^{-}\oplus \R u\ \text{ and }\ \hat{E}(u):=E^{-}\oplus [0,+\infty)u
$$ 
Hence
$$
E(u)=E^{-}\oplus \R u^{+}\ \text{ and }\ \hat{E}(u)=E^{-}\oplus [0,+\infty)u^{+}.
$$ 
Moreover, we fix the real numbers
$$
\gamma_W:=\inf_{\M}\Phi_{W} \quad \mbox{and} \quad \gamma:=\inf_{\M}\Phi.
$$

\begin{lemma}\label{contadomaximon2}
	If $u\in\M$ and $w=su+v$ where $s\geq1$ and $v\in E^{-}$ such that $w\neq0$, then
	$$
	\Phi_{W}(u+w)<\Phi_{W}(u)
	$$
\end{lemma}
\begin{proof}
The proof follows as in Lemma \ref{contadomaximo}.
\end{proof}

\begin{lemma}\label{lemacompacton2}
    Let $\K\subset E^{+}\setminus\{0\}$ be a compact subset, then there exists $R>0$ such that $\Phi_{W}(w)\leq0,\ \forall\ w\in E(u)\setminus B_{R}(0)$ and  $u\in\K$.
\end{lemma}
\begin{proof}
Fix $\lambda:=\inf_{x \in \R^2}D(x)$ and repeat the argument used in the proof of Lemma \ref{lemacompacto}.
\end{proof}

\begin{lemma}\label{supercontinuidaden2}
	For all $u\in H^{1}(\R^{2})$, the functional $\Phi_W|_{E(u)}$ is weakly upper semicontinuous.
\end{lemma}
\begin{proof}
See proof of Lemma \ref{supercontinuidade}. 
\end{proof}

\begin{lemma}\label{funcaomn2}
	For all $u\in H^{1}(\R^{2})\setminus E^{-}$, $\M\cap\hat{E}(u)$ is a singleton set and the element of this set is the unique global maximum of $\Phi_{W}|_{\hat{E}(u)}$
\end{lemma}
\begin{proof}
See proof of Lemma \ref{funcaom}.
\end{proof}

\vspace{0.5 cm}

In the proof of next lemma the fact that $f$ has an exponential critical growth brings some difficulty and we will do its proof. 

\begin{lemma}\label{primeirageometrian2}
There exists $\rho>0$ such that $\displaystyle\inf_{B_{\rho}(0)\cap E^{+}}\Phi_{W}>0$.
\end{lemma}
\begin{proof}
Given $p>2$ and $\epsilon>0$, there is $C_{\epsilon}>0$ such that
$$
|F(x,t)|\leq\epsilon|t|^{2}+C_{\epsilon}|t|^{p}(e^{4 \pi t^{2}}-1), \quad \forall (x,t) \in \R^2 \times \R. 
$$
Then, for all $u\in E^{+}$, the Lemmas \ref{lema 1.2} and \ref{lema 1.3} lead to 
$$
\begin{array}{ll}
\Phi_{W}(u)= & \frac{1}{2}||u||^{2}-\frac{1}{2}\int W(x)|u|^{2}dx-\int F(x,u)dx \\
& \geq\frac{1}{2}||u||^{2}-\frac{\Theta}{2}\int |u|^{2}dx-\epsilon\int|u|^{2}dx 
-C_{\epsilon}\int |u|^{p}(e^{4 \pi u^{2}}-1)dx \\
& =\frac{1}{2}||u||^{2}-\frac{\Theta}{2\overline{\Lambda}}|||u||^{2}-\frac{\epsilon}{\overline{\Lambda}}||u||^{2}-C_{\epsilon}|u|_{{2p}}^{p}\left(\int (e^{8 \pi u^{2}}-1)dx\right)^{\frac{1}{2}} \\
& \geq \left[\frac{1}{2}\left(1-\frac{\Theta}{\overline{\Lambda}}\right)-\frac{\epsilon}{\overline{\Lambda}}\right]||u||^{2}-C||u||^{p}\left(\int (e^{8 \pi u^{2}}-1)dx\right)^{\frac{1}{2}}.
\end{array}
$$
By Lemma \ref{lema 1.2}, if $\rho < \frac{\sqrt{3}}{2\sqrt{2}}$, 
$$
\sup_{\|u\|=\rho}\int (e^{8 \pi u^{2}}-1)dx \leq \sup_{\|v\|\leq 1}\int (e^{ 3 \pi u^{2}}-1)dx=C<\infty. 
$$
So, 
$$
\Phi_{W}(u)\geq\left[\frac{1}{2}\left(1-\frac{\Theta}{\overline{\Lambda}}\right)-\frac{\epsilon}{\overline{\Lambda}}\right]||u||^{2}-C||u||^{p}.
$$
Hence, decreasing $\rho$ if necessary and  fixing $\epsilon$ small enough, we get 
$$
\Phi_{W}(u)\geq\left[\frac{1}{2}\left(1-\frac{\Theta}{\overline{\Lambda}}\right)-\frac{\epsilon}{\overline{\Lambda}}\right]\rho^{2}-C\rho^{p}=\beta>0. 
$$
\end{proof}

\begin{lemma}\label{cpositivon2}
	The real number $\gamma_W$ is positive. In addition, if $u\in\M$ then $||u^{+}||\geq\max\{||u^{-}||,\sqrt{2 \gamma_W} \}$.
\end{lemma}
\begin{proof} See proof of Lemma \ref{cpositivo}
\end{proof}

\vspace{0.5 cm}

The next lemma shows that $(PS)$ sequences of $\Phi_W$ are bounded, as we are working with the exponential critical growth the arguments explored in Section 2 does not work in this case and a new proof must be done.

\begin{lemma}\label{limitacaosequenciapsn2}
	If $(u_{n})$ is a sequence such that  
	$$
	\Phi_{W}(u_{n})\leq d,\ \ \ \pm \Phi_{W}'(u_{n})u_{n}^{\pm}\leq d||u_{n}||\ \ \ \text{and}\ \ \ -\Phi_{W}'(u_{n})u_{n}\leq d
	$$
	for some $d>0$, then $(u_{n})$ is bounded in $H^{1}(\R^{2})$ and $(f(u_n)u_n)$ is bounded in $L^{1}(\R^2)$.
\end{lemma}
\begin{proof}
First of all, note that
$$
\left(\frac{1}{2}-\frac{1}{\theta}\right)\int f(x,u_{n})u_{n}dx\leq\Phi_W(u_{n})-\frac{1}{2}\Phi_W'(u_{n})u_{n}\leq 2d.
$$
Hence, $\left(\int f(x,u_{n})u_{n}dx\right)$ is bounded. Recalling that $f(x,t)t \geq 0$ for all $t \in\ \R$ and $x \in \mathbb{R}^N$, it follows that $(f(x,u_n)u_n)$ is bounded in $L^{1}(\R^N)$. On the other hand, we know that 
$$
||u_{n}^{+}||^{2}\leq d||u_{n}^{+}||+\int f(x,u_{n})u_{n}^{+}dx+\int W(x)u_{n}u_{n}^{+}dx
$$
and so,
\begin{equation}\label{desigualdadecommais}
||u_{n}^{+}||^{2}\leq d||u_{n}^{+}||+\left(\int f(x,u_{n})v_{n}dx\right)||u_{n}^{+}||_{H^{1}(\R^{N})}+\int W(x)u_{n}u_{n}^{+}dx
\end{equation}
where $v_{n}:=\frac{u_{n}^{+}}{||u_{n}^{+}||_{H^{1}(\R^{2})}}$. 

\begin{claim}
$\left(\int f(x,u_{n})v_{n}dx\right)$ is a bounded sequence.
\end{claim}

\noindent Indeed, by a direct computation, there exists $K>0$ such that
\begin{equation}\label{implicacao}
|f(x,t)|\leq Ce^{1/4} \Rightarrow |f(x,t)|^{2}\leq K f(x,t)t, \quad \mbox{uniformly in} \quad x.
\end{equation}
Moreover, by \cite[Lemma 2.11]{DORUF},
\begin{equation}\label{implicacao2}
rs\leq(e^{r^{2}}-1)+s(log^{+}s)^{1/2}+\frac{1}{4}s^{2}\chi_{[0,e^{1/4}]}(s) \quad \forall r,s \geq 0.
\end{equation}
Now, the Lemma \ref{lema 1.2} combined with the above  inequalities  for  $r=|v_{n}|$ and $s=\frac{1}{\Gamma }|f(u_{n})|$ leads to  
$$
\begin{array}{l}
\left|\int f(x,u_{n})v_{n}dx\right|\leq \Gamma\int\frac{1}{\Gamma}|f(u_{n})||v_{n}|dx\leq \Gamma\int(e^{v_{n}^{2}}-1)dx+ \\
+\int |f(x,u_{n})|\left(log^{+}\left(\frac{1}{\Gamma}|f(x,u_{n})|\right)\right)^{1/2}dx+ \\
\frac{1}{4\Gamma}\int |f(x,u_{n})|^{2}\chi_{[0,e^{1/4}]}\left(\frac{1}{\Gamma}|f(x,u_{n})|\right)dx\leq \\
\Gamma T+\int |f(x,u_{n})|\left(log^{+}\left(e^{4\pi u_{n}^{2}}\right)\right)^{1/2}dx+\frac{1}{4\Gamma}\int_{|f(x,u_{n})|\leq \Gamma e^{1/4}} |f(x,u_{n})|^{2}dx\leq \\
\Gamma T+\int |f(x,u_{n})||u_{n}|\sqrt{4\pi}dx+\frac{1}{4\Gamma}\int_{|f(x,u_{n})|\leq \Gamma e^{1/4}}Kf(x,u_{n})u_{n}dx.
\end{array}
$$
As $(f(x,u_n)u_n)$ is bounded in $L^{1}(\R^2)$, the last inequality yields $\left(\int f(x,u_{n})v_{n}dx\right)$ is bounded. Consequently, there exists $A_{0}>0$ satisfying 
$$
\left|\int f(x,u_{n})v_{n}dx\right|\leq A_{0} \quad \forall n \in \mathbb{N}.
$$
Thereby, by (\ref{desigualdadecommais}), 
\begin{equation}\label{desigmais}
||u_{n}^{+}||^{2}\leq d||u_{n}^{+}||+A_{0}||u_{n}^{+}||_{H^{1}(\R^{N})}+\int W(x)u_{n}u_{n}^{+}dx.
\end{equation}
Analogously, there is $B_{0}>0$ such that
\begin{equation}\label{desigmenos}
||u_{n}^{-}||^{2} \leq 
d||u_{n}^{-}||+B_{0}||u_{n}^{-}||_{H^{1}(\R^{N})}-\int W(x)u_{n}u_{n}^{-}dx.
\end{equation}
The inequalities (\ref{desigmais}) and (\ref{desigmenos}) combine to give 
$$
\begin{array}{l}
||u_{n}||^{2}\leq {C}||u_{n}||+{C}||u_{n}||+\int W(x)(u_{n}u_{n}^{+}-u_{n}u_{n}^{-})dx=2{C}||u_{n}||+\\
+\int W(x)((u_{n}^{+})^{2}-(u_{n}^{-})^{2})dx\leq 2{C}||u_{n}||+\int W(x)(u_{n}^{+})^{2}dx\leq 2{C}||u_{n}||+\frac{\Theta}{\overline{\Lambda}}||u_{n}^{+}||^{2}
\end{array}
$$
for some ${C}>0$. Hence,
$$
\left(1-\frac{\Theta}{\overline{\Lambda}}\right)||u_{n}||^{2}\leq 2\widetilde{C}||u_{n}||,
$$
from where it follows that $(u_{n})$ is bounded.
\end{proof}

\vspace{0.5 cm}

As a byproduct of the last lemma we have the corollary below

\begin{corollary}\label{corolariocoercividadeemMn2}
$\Phi_{W}$ is coercive on $\M$, that is, $\Phi_{W}(u)\rightarrow+\infty$ as $||u||\rightarrow+\infty,\ u \in \M$.
\end{corollary}

Moreover, we also have the following result
\begin{corollary} \label{LIM2}
	If $(u_{n})$ is a (PS) sequence for $\Phi_{W}$, then $(u_{n})$ is bounded. In addition, if $u_{n}\rightharpoonup u$ in $H^{1}(\R^{N})$, then $u$ is a solution of $(\ref{problema}).$
\end{corollary}	
\begin{proof} The corollary  follows applying the Lemmas \ref{lema 1.2} and \ref{LX}.    
\end{proof}	

\vspace{0.5 cm}

As in Section 2, the Lemma \ref{funcaomn2} permits to define a function 
$$
m:E^{+}\setminus\{0\}\rightarrow\M\ \text{ where }\ m(u)\in\hat{E}(u)\cap\M\ \ \forall u\in E^{+}\setminus\{0\}.
$$

Now, we invite the reader to observe that the same approach used in Section 2 works to guarantee that the proposition below holds

\begin{proposition}\label{pssequencen2}
	There exists a (PS)$_{\gamma_W}$ sequence for $\Phi_{W}$.
\end{proposition}

Our next proposition is crucial when $f$ has an exponential critical growth.  

\begin{proposition}\label{controlec} Fixed $\widetilde{A} \in (0,1/a)$ , there is $\lambda^*>0$  such that $\gamma_W<\frac{\widetilde{A}^{2}}{2}$ for $\inf_{\mathbb{R}^2}D (x)> \lambda^*$, where $a$ was given in (\ref{equivalente}).
\end{proposition}
\begin{proof}
Let $u\in E^{+}$ with $u \not=0$ and set 
$$
h_{D}(s):=As^{2}-\lambda Bs^{q}, 
$$ 
where 
$$
\lambda= \inf_{x \in \mathbb{R}^2}D(x), \; \;\; A=\frac{1}{2}||u||^{2} \quad \mbox{and} \quad B=\xi\int|u|^{q}dx,
$$ 
with $\xi$ given in Lemma \ref{projecaolp}. Then, a straightforward computation leads to  
$$
\max_{s \geq 0}h_{D}(s)=\left(A-\frac{2A}{q}\right)\left(\sqrt[q-2]{\frac{2A}{qB\lambda}}\right)^{2}.
$$
Thereby, by $(f_6)$ and Lemma \ref{projecaolp},
$$
\begin{array}{l}
c\leq\sup_{\scriptsize\begin{array}{c}  s\in[0,+\infty) \\ v\in E^{-} \end{array}}\Phi_{W}(su+v)=\sup_{\scriptsize\begin{array}{c}  ||su+v||\leq r \\ s\geq s_{0}, v\in E^{-} \end{array}}\Phi_{W}(su+v) \\
\leq\sup_{\scriptsize\begin{array}{c}  ||su+v||\leq r \\ s\geq s_{0}, v\in E^{-} \end{array}}\left[\frac{1}{2}s^{2}||u||^{2}-\int F(x,su+v)dx\right] \\
\leq \sup_{\scriptsize\begin{array}{c}  ||su+v||\leq r \\ s\geq s_{0}, v\in E^{-} \end{array}}\left[\frac{1}{2}s^{2}||u||^{2}-\lambda \int|su+v|^{q}dx\right]\\
\leq\sup_{\scriptsize\begin{array}{c}  ||su+v||\leq r \\ s\geq s_{0}, v\in E^{-} \end{array}}\left[\frac{1}{2}s^{2}||u||^{2}-\lambda \xi s^{q}\int|u|^{q}dx\right]\\
=\sup_{\scriptsize\begin{array}{c}  ||su+v||\leq r \\ s\geq s_{0}, v\in E^{-} \end{array}}h_{D}(s)  \\
\leq \max_{s \geq 0}h_{D}(s)= \left(A-\frac{2A}{q}\right)\left(\sqrt[q-2]{\frac{2A}{qB\lambda}}\right)^{2}.
\end{array}
$$
From the last inequality there is $\lambda^*>0$ such that 
$$
\left(A-\frac{2A}{q}\right)\left(\sqrt[q-2]{\frac{2A}{qB\lambda}}\right)^{2} < \frac{\widetilde{A}^2}{2}, \quad \forall \lambda \geq \lambda^*,
$$
finishing the proof. 
\end{proof}

\begin{proposition} \label{LIONS}
Fix $\inf_{x \in\R^2}D(x) \geq \lambda^*$ and $r>0$. Then, there exist a sequence $(y_{n}) \subset \R^{2}$ and $\eta>0$ such that 
$$
\int_{B_{r}(y_{n})}|u_{n}^{+}|^{2}dx\geq\eta>0,\ \ \ \forall n\in\N.
$$
Moreover, increasing $r$ if necessary, the sequence $(y_n)$ can be chosen in $\Z^2$.

\end{proposition}
\begin{proof}
Suppose by contradiction that the lemma does not hold for some $r>0$. Then, by a lemma due to Lions \cite{lionsI}, 
$$
u_{n}^{+}\to 0\ \ \text{in }L^{p}(\R^{2}),\ \forall\ p\in(2,+\infty).
$$
Define $w_{n}:=\widetilde{A}\frac{u_{n}^{+}}{||u_{n}^{+}||}$. Since $u_{n}\in\M$ for all $n\in\N$,  from Lemma \ref{cpositivon2} we have $\displaystyle\liminf_{n\in\N}||u_{n}^{+}||>0$, and so,
$$
w_{n}\to 0 \ \text{ in }L^{p}(\R^{2}),\ \forall\ p\in(2,+\infty).
$$
On the other hand, we also know that 
$$
||w_{n}||_{H^{1}(\R^{2})}=\widetilde{A}\frac{||u_{n}^{+}||_{H^{1}(\R^{2})}}{||u_{n}^{+}||}\leq\widetilde{A}a\frac{||u_{n}^{+}||}{||u_{n}^{+}||}=\widetilde{A}a<1.
$$
As $w_{n}\in\widehat{E}(u_{n})$ and $u_{n}\in\M$, we derive that \begin{equation}\label{desigualdadelimite}
\Phi(u_{n})\geq \Phi(w_{n})=\frac{1}{2}\widetilde{A}^{2}-\int F(x,w_{n})dx.
\end{equation}
By \cite[Proposition 2.3]{AOM}, we have $\int F(x,w_{n})dx \to 0$. Therefore, passing to the limit in (\ref{desigualdadelimite}) as $n \to +\infty$, we obtain 
$$
\gamma_W \geq\frac{\widetilde{A}^{2}}{2},
$$ 
which contradicts the Proposition \ref{controlec}.  Thus, there are $(z_n) \subset \R^2$ and $\eta>0$ such that 
$$
\int_{B_{r}(z_{n})}|u_{n}^{+}|^{2}dx\geq\eta>0,\ \ \ \forall n\in\N.
$$
Now, we repeat the same idea explored in Lemma \ref{finalperiodico} to conclude the proof. 
\end{proof}

\subsection{Proof of Theorem \ref{T1}: The case $N =2$. }

As in Section 2, the proof will be divided into two cases, the Periodic Case and  the  Asymptotically Periodic Case. 

\vspace{0.5 cm}

\subsection{Periodic Case}

\begin{proof}
First of all, we recall there is a $(PS)_{\gamma_W}$ sequence $(u_n)$ for $\Phi$ which must be bounded. Thus, there is $u \in H^{1}(\R^2)$ such that for some subsequence  of $(u_n)$, still denoted by itself, we have
$$
u_n \rightharpoonup u \quad \mbox{in} \quad H^1(\R^2)
$$
and
$$
u_n(x) \to u(x) \quad \mbox{a.e. in} \quad \R^2.
$$
Moreover, by Lemma \ref{limitacaosequenciapsn2} the sequence $(f(x,u_n)u_n)$ is bounded in $L^{1}(\R^2)$. Therefore, by  Lemma \ref{LX},
$$
\Phi'(u)\phi=0, \quad \forall \phi \in C^{\infty}_0(\R^2).
$$
If we combine the Lemma \ref{lema 1.2} with the density of $C^{\infty}_0(\R^2)$ in $H^{1}(\R^2)$, we see that  $u$ is a critical point of $\Phi$, that is,
$$
\Phi'(u)v=0, \quad \forall v \in H^1(\R^2).
$$ 
Moreover, by Fatou's Lemma, we also have
$$
\Phi(u) \leq \gamma.
$$
If $u \not=0$, we must have
$$
\Phi(u) \geq \gamma,
$$
showing that $\Phi(u)=\gamma$, and so, $u$ is a ground state solution.

If $u=0$, we can apply Lemma \ref{LIONS} to get a sequence $(y_n) \subset \Z^2$ and real numbers $r, \eta>0$ verifying 
$$
\int_{B_{r}(y_{n})}|u_{n}^{+}|^{2}dx\geq\eta>0,\ \ \ \forall n\in\N.
$$
Setting $v_n(x)=u_n(x+y_n)$, a direct computation gives that $(v_n)$ is also a $(PS)_\gamma$ for $\Phi$.  Moreover, for some subsequence, there is $v \in H^1(\R^2)$ such that  
$$
v_n \rightharpoonup v \quad \mbox{in} \quad H^{1}(\R^2) \quad \mbox{and} \quad \int_{B_{r}(0)}|v^{+}|^{2}dx\geq\eta>0,
$$ 
showing that $v\not=0$. Therefore, arguing as above, $v$ is a ground state solution for $\Phi$. \end{proof}

\subsection{The Asymptotically Periodic Case}

\begin{proof} First of all, we recall that $\Phi_{W}\leq\Phi$, and so, $\gamma_{W}\leq \gamma$. As in Section 2, we will consider the cases $\gamma_W=\gamma$ and $\gamma_W < \gamma$. The first one follows as in Section 2, and we will omit its proof.  

In what follows, we are considering $\gamma_{W}<\gamma $ and $(u_{n})$ be a $(PS)_{\gamma_{W}}$ sequence for $\Phi_{W}$ which was given in Lemma \ref{pssequencen2}. The sequence $(u_{n})$ is bounded by Lemma \ref{limitacaosequenciapsn2}. Thus, there is $u \in H^{1}(\R^2)$ and a subsequence of $(u_n)$, still denoted by itself, such that $u_{n}\rightharpoonup u$ in $H^{1}(\R^2)$. Suppose by contradiction $u=0$. Repeating the arguments explored in the case $N\geq3$, we have 
$$
\int W(x)|u_{n}|^{2}dx \to 0 \;\;\;\; \mbox{and} \;\;\;\; \sup_{\|\psi\|\leq 1}\left|\int W(x) u_{n}\psi dx\right| \to 0.
$$
From $(f_1)$, given $\epsilon>0$ and $\beta>0$ such that 
$$
\beta<\frac{2\pi}{\sup_{n\in\N}||u_{n}||^{2}},
$$ 
it must  exist  $\eta>0$ satisfying  
$$
|f^{*}(x,s)|\leq\epsilon(e^{\beta s^{2}}-1) \;\; \forall x\in\R^{2}\setminus B_{\eta}(0). 
$$
Therefore, by Lemma \ref{lema 1.2}
$$
\begin{array}{l}
\int_{|x|\geq\eta}|f^{*}(x,u_{n})||\psi|dx\leq \int_{|x|\geq\eta}\epsilon|e^{\beta u_{n}^{2}}-1||\psi|dx\leq \\
\leq\epsilon\left(\int_{|x|\geq\eta}|e^{\beta u_{n}^{2}}-1|^{2}dx\right)^{1/2}\left(\int_{|x|\geq\eta}|\psi|^{2}dx\right)^{1/2}dx\leq\epsilon K||\psi||_{H^{1}(\R^{2})},
\end{array}
$$
leading to  
$$
\sup_{\|\psi\| \leq 1}\left|\int f^{*}(x,u_{n})\psi  \, dx\right| \to 0.
$$
A similar argument works to prove that 
$$
0 \leq \int F^{*}(x,u_{n})dx\leq\int f^{*}(x,u_{n})u_{n}dx \to 0.
$$
The above limits yield
$$
\Phi(u_{n}) \to \gamma_{W} \;\;\; \mbox{and} \;\;\; ||\Phi'(u_{n})|| \to 0. 
$$
Arguing as in the periodic case, without loss of generality, we can assume that
$$
u_n \rightharpoonup u \quad \mbox{in} \quad H^1(\R^2), u \not=0 \;\; \mbox{and} \;\; \Phi'(u)=0.
$$
Thus, $\Phi(u) \geq \gamma$. On the other hand, by Fatou's Lemma,
$$
\Phi(u) \leq \liminf_{n\to +\infty}\Phi(u_n)=\gamma_W,
$$
which is absurd, because we are supposing $\gamma_W < \gamma$. Thereby, $u \not=0$ and since $(f(x,u_n)u_n)$ is bounded in $L^{1}(\R^2)$, we can conclude that $u$ is a ground state solution of $\Phi_W$. 
\end{proof}

\end{document}